\pdfoutput=1
\newif\ifpersonal
\documentclass[12pt]{amsart} 
\usepackage{amsmath,amsthm,amssymb,mathrsfs,mathtools,bm} 
\usepackage{microtype,fixltx2e,lmodern,textcomp} 
\usepackage{enumerate,comment,braket,xspace,tikz-cd} 
\usepackage[utf8]{inputenc} 
\usepackage[T1]{fontenc} 
\usepackage[centering,vscale=0.7,hscale=0.7]{geometry}
\usepackage[hidelinks,linktoc=all]{hyperref}
\usepackage[capitalize]{cleveref}

\theoremstyle{plain}
\newtheorem{thm}{Theorem}[section]
\newtheorem{thm-intro}[thm]{Theorem}
\newtheorem{lem}[thm]{Lemma}
\newtheorem{prop}[thm]{Proposition}
\newtheorem*{prop*}{Proposition}
\newtheorem{conj}[thm]{Conjecture}

\newtheorem{assumption}[thm]{Assumption}
\theoremstyle{definition}
\newtheorem{defin}[thm]{Definition}

\newtheorem{rem}[thm]{Remark}
\newtheorem{construction}[thm]{Construction}
\numberwithin{equation}{section}

\ifpersonal
\newcommand*{\personal}[1]{\textcolor[rgb]{0,0,1}{(Personal: #1)}}
\newcommand*{\todo}[1]{\textcolor{red}{(Todo: #1)}}
\else
\newcommand*{\personal}[1]{\ignorespaces}
\newcommand*{\todo}[1]{\ignorespaces}
\fi

\newcommand{\C}{\mathbb C}

\newcommand{\R}{\mathbb R}
\newcommand{\Z}{\mathbb Z}

\newcommand{\rb}{\mathrm b}
\newcommand{\rd}{\mathrm d}
\newcommand{\rf}{\mathrm f}
\newcommand{\ri}{\mathrm i}
\newcommand{\rn}{\mathrm n}
\newcommand{\rp}{\mathrm p}

\newcommand{\fC}{\mathfrak C}

\newcommand{\fX}{\mathfrak X}

\newcommand{\ff}{\mathfrak f}

\newcommand{\cA}{\mathcal A}
\newcommand{\cB}{\mathcal B}
\newcommand{\cC}{\mathcal C}

\newcommand{\cM}{\mathcal M}

\newcommand{\cO}{\mathcal O}

\newcommand{\cY}{\mathcal Y}

\DeclareFontFamily{U}{BOONDOX-calo}{\skewchar\font=45 }
\DeclareFontShape{U}{BOONDOX-calo}{m}{n}{<-> s*[1.05] BOONDOX-r-calo}{}
\DeclareFontShape{U}{BOONDOX-calo}{b}{n}{<-> s*[1.05] BOONDOX-b-calo}{}
\DeclareMathAlphabet{\mathcalboondox}{U}{BOONDOX-calo}{m}{n}

\newcommand{\bbA}{\mathbb A}

\newcommand{\bbP}{\mathbb P}

\newcommand{\bB}{\mathbf B}

\newcommand{\be}{\mathbf e}

\newcommand{\sX}{\mathscr X}

\makeatletter
\let\save@mathaccent\mathaccent
\newcommand*\if@single[3]{%
	\setbox0\hbox{${\mathaccent"0362{#1}}^H$}%
	\setbox2\hbox{${\mathaccent"0362{\kern0pt#1}}^H$}%
	\ifdim\ht0=\ht2 #3\else #2\fi
}
\newcommand*\rel@kern[1]{\kern#1\dimexpr\macc@kerna}
\newcommand*\widebar[1]{\@ifnextchar^{{\wide@bar{#1}{0}}}{\wide@bar{#1}{1}}}
\newcommand*\wide@bar[2]{\if@single{#1}{\wide@bar@{#1}{#2}{1}}{\wide@bar@{#1}{#2}{2}}}
\newcommand*\wide@bar@[3]{%
	\begingroup
	\def\mathaccent##1##2{%
		\let\mathaccent\save@mathaccent
		\if#32 \let\macc@nucleus\first@char \fi
		\setbox\z@\hbox{$\macc@style{\macc@nucleus}_{}$}%
		\setbox\tw@\hbox{$\macc@style{\macc@nucleus}{}_{}$}%
		\dimen@\wd\tw@
		\advance\dimen@-\wd\z@
		\divide\dimen@ 3
		\@tempdima\wd\tw@
		\advance\@tempdima-\scriptspace
		\divide\@tempdima 10
		\advance\dimen@-\@tempdima
		\ifdim\dimen@>\z@ \dimen@0pt\fi
		\rel@kern{0.6}\kern-\dimen@
		\if#31
		\overline{\rel@kern{-0.6}\kern\dimen@\macc@nucleus\rel@kern{0.4}\kern\dimen@}%
		\advance\dimen@0.4\dimexpr\macc@kerna
		\let\final@kern#2%
		\ifdim\dimen@<\z@ \let\final@kern1\fi
		\if\final@kern1 \kern-\dimen@\fi
		\else
		\overline{\rel@kern{-0.6}\kern\dimen@#1}%
		\fi
	}%
	\macc@depth\@ne
	\let\math@bgroup\@empty \let\math@egroup\macc@set@skewchar
	\mathsurround\z@ \frozen@everymath{\mathgroup\macc@group\relax}%
	\macc@set@skewchar\relax
	\let\mathaccentV\macc@nested@a
	\if#31
	\macc@nested@a\relax111{#1}%
	\else
	\def\gobble@till@marker##1\endmarker{}%
	\futurelet\first@char\gobble@till@marker#1\endmarker
	\ifcat\noexpand\first@char A\else
	\def\first@char{}%
	\fi
	\macc@nested@a\relax111{\first@char}%
	\fi
	\endgroup
}
\makeatother

\newcommand{\oGamma}{\widebar\Gamma}

\newcommand{\oOmega}{\widebar\Omega}

\newcommand{\otau}{\widebar\tau}
\newcommand{\oB}{\widebar B}

\newcommand{\oM}{\widebar M}
\newcommand{\oR}{\widebar R}
\newcommand{\oS}{\widebar S}
\newcommand{\oU}{\widebar U}

\newcommand{\oc}{\widebar c}
\newcommand{\oh}{\widebar h}

\newcommand{\hh}{\widehat h}

\newcommand{\hL}{\widehat L}

\newcommand{\hbeta}{\widehat\beta}
\newcommand{\hGamma}{\widehat\Gamma}

\newcommand{\tA}{\widetilde A}
\newcommand{\tB}{\widetilde B}
\newcommand{\tD}{\widetilde D}

\newcommand{\tL}{\widetilde L}

\newcommand{\tU}{\widetilde U}

\newcommand{\tY}{\widetilde Y}
\newcommand{\tbeta}{\widetilde{\beta}}

\newcommand{\tDelta}{\widetilde{\Delta}}

\newcommand{\ttau}{\widetilde{\tau}}

\newcommand{\sm}{\mathrm{sm}}

\newcommand{\bcM}{\widebar{\mathcal M}}

\newcommand{\Gmk}{\mathbb G_{\mathrm m/k}}
\newcommand{\GmC}{\mathbb G_{\mathrm m/\mathbb C}}

\newcommand{\llp}{(\!(}
\newcommand{\rrp}{)\!)}
\newcommand{\an}{^\mathrm{an}}

\newcommand{\ev}{\mathrm{ev}}

\newcommand{\inv}{^{-1}}
\newcommand{\id}{\mathrm{id}}

\newcommand{\kanal}{$k$-analytic\xspace}

\newcommand{\Zaffine}{$\mathbb Z$-affine\xspace}

\newcommand{\trop}{^\mathrm{trop}}

\newcommand{\DM}{Deligne-Mumford\xspace}
\providecommand{\abs}[1]{\lvert#1\rvert}
\providecommand{\norm}[1]{\lVert#1\rVert}

\usetikzlibrary{decorations.markings} 
\tikzset{
  closed/.style = {decoration = {markings, mark = at position 0.5 with { \node[transform shape, xscale = .8, yscale=.4] {/}; } }, postaction = {decorate} },
  open/.style = {decoration = {markings, mark = at position 0.5 with { \node[transform shape, scale = .7] {$\circ$}; } }, postaction = {decorate} }
}

\DeclareMathOperator{\Hom}{Hom}

\DeclareMathOperator{\NE}{NE}

\DeclareMathOperator{\Pic}{Pic}

\DeclareMathOperator{\rank}{rank}

\DeclareMathOperator{\Spec}{Spec}

\DeclareMathOperator{\st}{st}

\DeclareMathOperator{\val}{val}

\begin{document}
\title[Enumeration of holomorphic cylinders. II.]{Enumeration of holomorphic cylinders\\ in log Calabi-Yau surfaces. II.\\ Positivity, integrality and the gluing formula}
\author{Tony Yue YU}
\address{Tony Yue YU, Laboratoire de Mathématiques d'Orsay, Université Paris-Sud, CNRS, Université Paris-Saclay, 91405 Orsay, France}
\email{yuyuetony@gmail.com}
\date{August 3, 2016 (Revised on December 28, 2019)}

\begin{abstract}
We prove three fundamental properties of counting holomorphic cylinders in log Calabi-Yau surfaces: positivity, integrality and the gluing formula.
Positivity and integrality assert that the numbers of cylinders, defined via virtual techniques, are in fact nonnegative integers.
The gluing formula roughly says that cylinders can be glued together to form longer cylinders, and the number of longer cylinders equals the product of the numbers of shorter cylinders.
Our approach uses Berkovich geometry, tropical geometry, deformation theory and the ideas in the proof of associativity relations of Gromov-Witten invariants by Maxim Kontsevich.
These three properties provide an evidence for a conjectural relation between counting cylinders and the broken lines of Gross-Hacking-Keel.
\end{abstract}

\maketitle

\tableofcontents

\section{Introduction}

The motivation of this work comes from the study of mirror symmetry.
Mirror symmetry is a conjectural duality between Calabi-Yau manifolds.
A particular type of torus fibration, called \emph{SYZ fibration}, plays an important role in the subject (see \cite{SYZ_1996}).
The enumerative geometry of an SYZ fibration associated to a Calabi-Yau manifold is intimately related to the construction of its mirror manifold.
More precisely, we are interested in counting open holomorphic curves with boundaries on the torus fibers of SYZ fibrations.
Remarkable progress was made in this direction, notably by Fukaya, Oh, Ohta, Ono and many others (see \cite{Fukaya_Lagrangian_intersection_2009,Fukaya_Lagrangian_Floer_2012}).
However, there is a fundamental difficulty in the story:
One observes that in general, the numbers from open curve counting are not invariants.
They depend on various choices.

In order to overcome this difficulty, in \cite{Yu_Enumeration_cylinders_2015}, we explored a new approach using non-archimedean geometry.
The counting of holomorphic cylinders was obtained for log Calabi-Yau surfaces.
The key idea was to relate the counting of cylinders to the counting of particular types of rational curves.
Since the approach was rather indirect, and went through the mysterious non-archimedean geometry,
it is a priori unclear whether we have obtained the \emph{right} invariants.

In this paper, we prove three fundamental properties of these invariants, indicating that these invariants really reflect open curve counting.
The three properties are positivity, integrality and the gluing formula.
Positivity and integrality assert that the numbers of cylinders, defined via virtual techniques, are in fact nonnegative integers.
The gluing formula roughly says that cylinders can be glued together to form longer cylinders, and the number of longer cylinders equals the product of the numbers of shorter cylinders.
In fact, we count cylinders with boundaries on fibers of (non-archimedean) SYZ torus fibrations.
Given three torus fibers $F_1, F_2, F_3$, we are interested in counting cylinders going from $F_1$ to $F_2$, going from $F_2$ to $F_3$, and going from $F_1$ to $F_3$.
The gluing formula claims that the number of cylinders going from $F_1$ to $F_3$ passing through $F_2$ is equal to the product of the number of cylinders going from $F_1$ to $F_2$ and the number of cylinders going from $F_2$ to $F_3$.

\bigskip
Let us now give the precise statements of our theorems.

Let $(Y,D)$ be a Looijenga pair, i.e.\ a connected smooth complex projective surface $Y$ together with a singular nodal curve $D$ representing the anti-canonical class $-K_Y$.
Let $X\coloneqq Y\setminus D$.
Assume that the pair $(Y,D)$ is positive\footnote{This terminology was introduced in \cite{Gross_Mirror_Log_published}.}, in the sense that the intersection matrix $(D_i\cdot D_j)$ of the components of $D$ is not negative semi-definite, a particular case of which is when $X$ is affine.
Let $k\coloneqq\C\llp t\rrp$ be the field of formal Laurent series.
Let $X_k\coloneqq X\otimes_\C k$, and $X\an_k$ the analytification in the sense of Berkovich (see \cite{Berkovich_Spectral_1990}).
We have a non-archimedean SYZ fibration $\tau\colon X\an_k\to B$, where $B$ is a singular \Zaffine manifold homeomorphic to $\R^2$.

A \emph{spine} $L$ in $B$ is a triple $[\Gamma,(v_1,v_2),h\colon\Gamma\to B]$ where $\Gamma$ is a chain of segments, $(v_1,v_2)$ are the two endpoints of the chain, and $h$ is a piecewise linear map.
Here we omit some technical conditions on $L$ (see \cref{def:spine}).
Given a spine $L$ in $B$ and a curve class $\beta\in\NE(Y)$, we constructed in \cite{Yu_Enumeration_cylinders_2015} the number of holomorphic cylinders $N(L,\beta)$ associated to $L$ and $\beta$.

\begin{thm}[Positivity and integrality, see \cref{thm:degree}]\label{thm:positivity_integrality}
	The number $N(L,\beta)$ of holomorphic cylinders associated to the spine $L$ and the curve class $\beta$ is a nonnegative integer.
\end{thm}

In order to state the gluing formula, we start with two spines
\begin{align*}
& L^1 = [\Gamma^1, (v^1_1, v^1_2),\ h^1\colon\Gamma^1\to B],\\
& L^2 = [\Gamma^2, (v^2_1, v^2_2),\ h^2\colon\Gamma^2\to B].
\end{align*}
Let $e^1_2$ denote the edge of $\Gamma^1$ connected to the vertex $v^1_2$;
let $e^2_1$ denote the edge of $\Gamma^2$ connected to the vertex $v^2_1$.
Assume $h^1(v^1_2)=h^2(v^2_1)$ and $w_{v^1_2}(e^1_2)+w_{v^2_1}(e^2_1)=0$,
where $w_\cdot(\cdot)$ denotes the weight vector, i.e.\ derivative along the edge.
So we can glue $L^1$ and $L^2$ at the vertices $v^1_2$ and $v^2_1$, and form a new spine $L^3$ (see \cref{fig:L3}).
Let $\beta^3\in\NE(Y)$ be a curve class.

\begin{figure}[!ht]
	\centering
	\setlength{\unitlength}{0.5\textwidth}
	\begin{picture} (1,0.32)
	\put(0,0){\includegraphics[width=\unitlength]{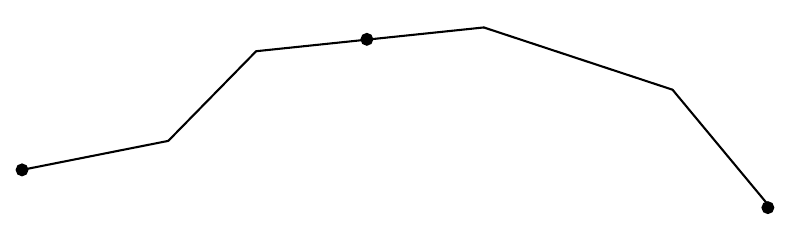}}
	\put(0.0,0.12){$v^1_1$}
	\put(0.44,0.275){$v^1_2$}	
	\put(0.36,0.262){$e^1_2$}	
	\put(0.44,0.18){$v^2_1$}	
	\put(0.52,0.278){$e^2_1$}	
	\put(0.99,0.01){$v^2_2$}
	\put(0.2,0.18){$L^1$}
	\put(0.72,0.24){$L^2$}
	\put(0.5,0.11){$L^3$}
	\end{picture}
	\caption{The spines $L^1$, $L^2$ and $L^3$.}
	\label{fig:L3}
\end{figure}

\begin{thm}[Gluing formula]\label{thm:gluing_formula}
	We have the following gluing formula:
	\[ \sum_{\beta^1+\beta^2=\beta^3} N(L^1,\beta^1)\cdot N(L^2,\beta^2) = N(L^3,\beta^3) .\]
\end{thm}

We use a combination of techniques from tropical geometry, Berkovich geometry, deformation theory and Gromov-Witten theory.
The proof of the positivity and integrality is based on three main ideas: First, the cylinder counts are invariant under deformations of the Looijenga pair $(Y,D)$; Second, for generic $(Y,D)$ in the deformation space, no bubbling can occur for our holomorphic cylinders in question; Third, by the deformation theory of rational curves in surfaces (see \cite{Keel_Rational_curves_1999}), we deduce that the moduli space responsible for our counting is generically smooth and of expected dimension.

Now we sketch the proof of the gluing formula.
We extend both $L^1$ and $L^2$ to infinity straight with respect to the \Zaffine structure, and glue them together to form a ``double extended spine'' $L^\rd$ as in \cref{const:extended_spines} \cref{fig:Ld}.
Then we consider the moduli stack of 5-pointed rational non-archimedean analytic stable maps associated to $L^\rd$.
The moduli space of 5-pointed rational curves has two special 0-strata corresponding to $G$ and $G'$ in \cref{const:graphs} \cref{fig:graphs}.
Then using a smoothness argument as in the proof of positivity and integrality, we prove that the counts of 5-pointed rational stable maps associated to $G$ and $G'$ are in fact equal.
Finally we explain in \cref{sec:gluing} that the two counts give exactly the two sides of the gluing formula.

In order to make the idea above precise, we need to construct carefully various moduli spaces of tropical curves and the corresponding moduli spaces of analytic curves, and study their properties, which constitutes the bulk of \cref{sec:tropical_constructions,sec:analytic_constructions}.
We will rely on the work of Abramovich-Caporaso-Payne \cite{Abramovich_Tropicalization_2012} and the author's previous work \cite{Yu_Tropicalization_2014}.
Our proof of the gluing formula via different degenerations of domain curves can be seen as a refinement of the proof of the associativity relations of Gromov-Witten invariants (see \cite{Kontsevich_Gromov-Witten_1994,Gathmann_Kontsevich_formula_2008}).
Comparing with the algebraic case, we need to take extra care of the properness of various moduli spaces in the proofs.

The results of this paper give us a deeper understanding of open curve counting of log Calabi-Yau surfaces.
They provide evidence that our cylinder counts obtained via Berkovich geometry satisfy expected properties and reflect honest open curve enumeration.
The counting of cylinders is intimately related to the combinatorial notion of broken line in the work of Gross-Hacking-Keel \cite{Gross_Mirror_Log_published}\footnote{The notion of broken line in general was developed by Gross, Hacking, Keel, Siebert and their coauthors in a series of papers \cite{Gross_Mirror_symmetry_for_P2_2010,Carl_A_tropical_view_2010,Gross_Mirror_Log_published,Gross_Theta_2012,Gross_Theta_2015}. It was also suggested by Abouzaid, Kontsevich and Soibelman in different occasions}.
The positivity, integrality and the gluing formula proved in this paper suggest a precise conjectural relation with broken lines which we state below:

Let $\gamma$ be a broken line in $B$ for the canonical scattering diagram in the sense of \cite[\S 2.3]{Gross_Mirror_Log_published}.
We follow the notations in loc.\ cit..
Let $\alpha$ be any negative number such that $(-\infty,\alpha]$ is a proper subset of the unique unbounded domain of linearity of $\gamma$.
Then the restriction of $\gamma$ to the interval $[\alpha,0]$ gives rise to a spine in $B$, which we denote by $L_\gamma$.
Let $c\cdot z^q$ denote the monomial attached to the last domain of linearity of $\gamma$.
Write $q=\beta+\lambda$ for $\beta\in\NE(Y)$ and $\lambda\in\Lambda_{\gamma(0)}$ (see \cite[Construction 2.2 and Example 2.3]{Gross_Mirror_Log_published}).

\begin{conj}
	The number $N(L_\gamma,\beta)$ of holomorphic cylinders associated to the spine $L_\gamma$ and the curve class $\beta$ is equal to the number $c$.
\end{conj}

Note that it is easy to verify the conjecture in simple cases, such as for toric $(Y,D)$.
Moreover, we would like to mention that the number $c$ is known to be a nonnegative integer via indirect combinatorial arguments (see \cite{Gross_Canonical_bases_2014}).
If the conjecture above holds, then \cref{thm:positivity_integrality} gives a geometric understanding of this fact.

\bigskip
\paragraph{\textbf{Related works}}
Besides the works already cited above, the following works are related to this paper in a broader context:  the works of Kontsevich-Soibelman \cite{Kontsevich_Homological_2001,Kontsevich_Affine_2006}, Auroux \cite{Auroux_Mirror_symmetry_anticanonical_2007,Auroux_Special_Lagrangian_fibrations_2009}, Gross-Siebert \cite{Gross_Real_Affine_2011}, Abouzaid \cite{Abouzaid_Family_2014} and Tu \cite{Tu_On_the_reconstruction_2014} on mirror symmetry, the works of Pandharipande-Solomon \cite{Pandharipande_Disk_enumeration_2008}, Welschinger \cite{Welschinger_Open_GW_2013}, Chan-Lau \cite{Chan_Open_GW_2014}, Fang-Liu \cite{Fang_Open_GW_2013}, Chan \cite{Chan_Open_GW_2012}, Nishinou \cite{Nishinou_Disk_2012}, Lin \cite{Lin_Open_GW_2014}, Brini-Cavalieri-Ross \cite{Brini_Crepant_resolutions_2013} and Ross \cite{Ross_Localization_and_gluing_of_orbifold_amplitudes_2014} on open curve enumeration, and the works of Chen-Satriano \cite{Chen_Chow_quotients_2013} and Ranganathan \cite{Ranganathan_Skeletons_I_2015} on stable maps in logarithmic geometry.
The technique of imposing tropical conditions on the moduli space of algebraic stable maps can also be achieved using logarithmic structures, see the works of Ranganathan-Santos-Parker-Wise \cite{Ranganathan_Moduli_of_stable_maps_in_genus_one_I,Ranganathan_Moduli_of_stable_maps_in_genus_one_II}.

\bigskip
\paragraph{\textbf{Acknowledgements}}
I am very grateful to Maxim Kontsevich for sharing with me many ideas.
I am equally grateful to Vladimir Berkovich, Antoine Chambert-Loir, Mark Gross, Bernd Siebert and Michael Temkin for valuable discussions.
The smoothness argument in \cref{sec:smoothness} I learned from Sean Keel.
I would like to thank him in particular.
This research was partially conducted during the period the author served as a Clay Research Fellow.

\section{Review of the counting of cylinders}\label{sec:review_of_cylinder_counts}

In this section, we review the construction of the counting of holomorphic cylinders in log Calabi-Yau surfaces following \cite{Yu_Enumeration_cylinders_2015}.

We start with a Looijenga pair $(Y,D)$, i.e.\ a connected smooth complex projective surface $Y$ together with a singular nodal curve\footnote{A singular nodal curve is a singular curve with at worst ordinary double point singularities.
	It is not necessarily irreducible.} $D$ representing the anti-canonical class $-K_Y$.
Let $X\coloneqq Y\setminus D$.
It is called a log Calabi-Yau surface.

Note that our initial data are all defined over the complex numbers.
We will now use non-archimedean geometry as a tool in order to extract geometric invariants from $X$.

Let $k\coloneqq\C\llp t\rrp$ be the field of formal Laurent series.
Let $X_k\coloneqq X\otimes_\C k$, and let $X_k\an$ be the analytification in the sense of Berkovich \cite{Berkovich_Spectral_1990}.

We have an essential skeleton $B\subset X_k\an$ and a canonical retraction map $\tau\colon X\an_k\to B$.
As piecewise linear space, $B$ is isomorphic to $\R^2$.
Outside the origin $O\in B$, the map $\tau$ is an affinoid torus fibration\footnote{Recall that a continuous map from a \kanal space to a topological space is called an \emph{affinoid torus fibration} if locally on the target, it is isomorphic to the pullback of the coordinate-wise valuation map $\val^n\colon(\Gmk\an)^n\to\R^n$ along an open subset $U\subset\R^n$.} (see \cite[Proposition 3.6]{Yu_Enumeration_cylinders_2015}, see also \cite{Nicaise_Xu_Yu_The_non-archimedean_SYZ_fibration}).
Our goal is to study the counting of holomorphic cylinders with boundaries on affinoid fiber tori.

Consider a holomorphic cylinder in $X\an_k$, i.e.\ a closed annulus $C$ with two boundary points $v_1, v_2$, and a holomorphic map $f\colon C\to X\an_k$.
Let $\Gamma$ be the skeleton of $C$, i.e.\ the path in $C$ connecting $v_1$ and $v_2$.
Let $h\coloneqq (\tau\circ f)|_{\Gamma}$.
It is a piecewise linear map from $\Gamma$ to $B$.
We call $(\Gamma,(v_1,v_2),h)$ the spine associated to the holomorphic cylinder.

Now fix a spine $L$ (in the sense of \cref{def:spine}) and a curve class $\beta\in\NE(Y)$; we would like to count holomorphic cylinders in $X\an_k$ whose associated spine is $L$ and whose associated curve class is $\beta$.
Unfortunately, the moduli space of such holomorphic cylinders is infinite dimensional.
In order to obtain a finite number, our strategy in \cite{Yu_Enumeration_cylinders_2015} is to restrict to cylinders whose boundaries satisfy an extra regularity condition:
we require that when we make analytic continuation at the boundaries, our cylinders should extend straight to infinity in an appropriate sense.
In this way, we are able to relate the counting of cylinders to the counting of particular types of rational curves.

The way to carry out the idea above is the following:
First, using the \Zaffine structure on $B\setminus O$ induced by the non-archimedean SYZ fibration, we extend the spine $L$ at both ends straight to an infinite spine $\hL$, which we call the associated extended spine (see \cref{const:extension}).

Put $\tB\coloneqq\R\times B$.
The graph of the extended spine $\hL$ gives an extended spine $\tL$ in $\tB$.
We also have an associated curve class $\tbeta$ whose definition we omit here for simplicity (see \cite[\S 5]{Yu_Enumeration_cylinders_2015}).
Let $\tY$ be a toroidal compactification of $\Gmk\times X_k$ which has two boundary divisors $\tD_1, \tD_2$ corresponding to the two infinite directions of $\tL$.
Let $\ttau\colon(\Gmk\times X_k)\an\to \tB$ be the map induced by $\tau\colon X\an_k\to B$.
Fix a point $p$ in $(\Gmk\times X_k)\an$ such that $\ttau(p)\in\tB$ is the point in the image of $\tL$ corresponding to an endpoint of $L$.
Let $\cM$ be the stack of analytic stable maps into $\tY\an$ of class $\tbeta$ with three marked points, such that the first marked point maps to $\tD\an_1$, the second marked point maps to $\tD\an_2$, and the third marked point maps to $p$.
For any stable map $[C,(s_1,s_2,s_3),f\colon C\to\tY\an]$ in $\cM$, let $\Gamma$ be the path in $C$ connecting $s_1$ and $s_2$, and let $h\coloneqq (\ttau\circ f)|_{\Gamma\setminus\{s_1,s_2\}}$.
We call $(\Gamma,(s_1,s_2),h)$ the extended spine in $\tB$ associated to the stable map.

Let $\cM(\tL)$ be the substack of $\cM$ consisting of stable maps whose associated extended spine equals $\tL$.
It is proved in \cite[Theorem 5.4]{Yu_Enumeration_cylinders_2015} that $\cM(\tL)$ is a proper analytic stack.
Moreover, it possesses a virtual fundamental class of dimension zero.
Finally, we define the number of holomorphic cylinders $N(L,\beta)$ associated to the spine $L$ and the curve class $\beta$ to be the degree of this virtual fundamental class.
We have shown in loc.\ cit.\ that the number $N(L,\beta)$ is well-defined, i.e.\ it does not depend on the various choices we have made during the construction, in particular the choice of the point $p$.

Above is a sketch of the construction of the number $N(L,\beta)$.
We remark that although $N(L,\beta)$ is supposed to represent the number of cylinders, its actual construction involves enumeration of rational curves and virtual fundamental classes.
Therefore, the positivity and integrality in \cref{thm:positivity_integrality} and the gluing formula in \cref{thm:gluing_formula} are not at all obvious.
In order to prove these theorems, we will start by a series of constructions in tropical geometry.
This is the subject of the next section.

\section{Tropical constructions}\label{sec:tropical_constructions}

In this section, we will make a series of constructions concerning tropical curves and tropical moduli spaces.
Here is an overview:

Definitions \ref{def:unbounded_tree}, \ref{def:Zaffine_unbounded_tree} and \ref{def:Zaffine_map} set up the most basic notions regarding trees and \Zaffine maps in our context.
Definition \ref{def:spine} introduces the notion of spine, the fundamental combinatorial object associated to holomorphic cylinders.
As explained in \cref{sec:review_of_cylinder_counts}, for the enumeration of holomorphic cylinders, we introduce the notion of extended spine in \cref{def:extended_spine}, and the extension procedure in \cref{const:extension}.
\Cref{const:extended_spines} applies the extension procedure to the various spines in the statement of \cref{thm:gluing_formula}, and fix the notations for later use.

Constructions \ref{const:tB} and \ref{const:CB} concern subdivisions of the tropical bases, which will serve us for toric blowups in \cref{sec:analytic_constructions}.
Here is the idea behind \cref{const:CB}:
In \cref{sec:analytic_constructions}, it will be difficult to work directly with the tropical base $B$ which is unbounded.
So we approximate $B$ by a bounded subset $B_\rf\subset B$.
The conditions in \cref{const:CB} ensure that the approximation is good enough, so that it captures enough information about tropical curves in $B$ that are related to the spines $L^i$.

Constructions \ref{const:graphs} and \ref{const:Delta} define a path $\Delta\subset\oM\trop_{0,5}$ over which we will deform analytic curves.
Constructions \ref{const:subset_Ti} and \ref{const:subset_Td} are based on \cref{const:extended_spines}.
They define the sets of tropical curves in $B$ associated to the given spines.
As explained above, we need to work with an approximation $B_\rf\subset B$ later in \cref{sec:analytic_constructions}, so we make Constructions \ref{const:tropical_moduli} and \ref{const:TiTd} to select out subsets of tropical curves in $B_\rf$ associated to the given spines.
The main property of these subsets is given in \cref{prop:tropical_rigidity}, which paves the way for the properness results in \cref{sec:analytic_constructions}.

\begin{defin}\label{def:unbounded_tree}
	An \emph{unbounded tree} $(\Gamma,V_\infty(\Gamma))$ consists of a finite tree\footnote{Here ``finite'' means finitely many vertices.
	We allow 2-valent vertices.} $\Gamma$ and a subset of 1-valent vertices $V_\infty(\Gamma)$ called \emph{unbounded vertices}.
	Vertices that are not unbounded are called \emph{bounded vertices}.
	We require that there is at least one bounded vertex.
\end{defin}

\begin{defin}\label{def:Zaffine_unbounded_tree}
	An \emph{unbounded \Zaffine tree} $(\Gamma,V_\infty(\Gamma))$ consists of an unbounded tree and
	\begin{enumerate}
		\item for every edge $e$ with two bounded endpoints, a \Zaffine structure\footnote{We refer to \cite[\S 2]{Yu_Enumeration_cylinders_2015} for the definition of \Zaffine structure.} on $e$ which is isomorphic to the interval $[0,\alpha]\subset\R$ for a positive real number $\alpha$,
		\item for every edge $e$ with one unbounded endpoint $v_\infty$, a \Zaffine structure on $e\setminus\{v_\infty\}$ which is isomorphic to the interval $[0,+\infty)\subset\R$.
	\end{enumerate}
	We set $\Gamma^\circ\coloneqq\Gamma\setminus V_\infty(\Gamma)$.
	We will simply say a \emph{\Zaffine tree} when unbounded vertices are not present.
\end{defin}

\begin{rem}
	\cref{def:Zaffine_unbounded_tree} differs slightly from \cite[Definition 4.2]{Yu_Enumeration_cylinders_2015} in the sense that here we require that the tree is connected and contains at least one bounded vertex.
	The other cases are useless for our purpose.
\end{rem}

\begin{rem}
	Note that a one-dimensional \Zaffine structure is equivalent to a metric.
	So an unbounded \Zaffine tree without bounded 1-valent vertices is a rational tropical curve in the sense of \cite[\S 4.1]{Abramovich_Tropicalization_2012}.
	Our unbounded vertices correspond to the legs in loc.\ cit..
\end{rem}

\begin{defin} \label{def:Zaffine_map}
	Let $(\Gamma, V_\infty(\Gamma))$ be an unbounded \Zaffine tree.
	Let $(\Sigma,\Sigma_0)$ be a polyhedral \Zaffine manifold with singularities, i.e., a polyhedral complex $\Sigma$ equipped with a piecewise \Zaffine structure, an open subset $\Sigma_0\subset\Sigma$ which is a manifold without boundary, and a \Zaffine structure on $\Sigma_0$ compatible with the piecewise \Zaffine structure on $\Sigma$ (see \cite[\S 2]{Yu_Enumeration_cylinders_2015}).
	A \Zaffine map $h\colon\Gamma\setminus V_\infty(\Gamma)\to\Sigma$ is a continuous proper map such that every edge maps to a cell in $\Sigma$ compatibly with the \Zaffine structures.
	A \Zaffine immersion is a \Zaffine map that does not contract any edge to a point\footnote{A \Zaffine immersion is not an injective map in general.}.
	Given a \Zaffine map $h\colon\Gamma\setminus V_\infty(\Gamma)\to\Sigma$, a bounded vertex $v\in V(\Gamma)$ mapping to $\Sigma_0$, and an edge $e$ connected to $v$, we denote by $w^0_v(e)$ the primitive integral tangent vector at $v$ pointing to the direction of $e$, and by $w_v(e)$ the image of $w^0_v(e)$ in $T_{h(v)}\Sigma(\Z)$, the integral lattice in the tangent space $T_{h(v)}\Sigma$.
	We call $w_v(e)$ the weight vector of the edge $e$ at $v$.
	The \Zaffine map $h$ is said to be \emph{balanced} at the vertex $v$ if $\sum_{e\ni v} w_v(e)=0\in T_{h(v)}\Sigma(\Z)$, where the sum is taken over all edges of $\Gamma$ connected to $v$.
\end{defin}

\begin{rem}
	Let $B$ be as in \cref{sec:review_of_cylinder_counts}, i.e.\ the base of the non-archimedean SYZ fibration $\tau\colon X\an_k\to B$.
	By construction it is homeomorphic to the dual intersection cone complex of $D$.
	Let $O\in B$ be the origin in $B$.
	The non-archimedean SYZ fibration induces a \Zaffine structure on $B\setminus O$ (see \cite[Remark 3.7]{Yu_Enumeration_cylinders_2015}).
	So the pair $(B, B\setminus O)$ is a polyhedral \Zaffine manifold with singularities.
\end{rem}

\begin{defin}\label{def:spine}
	A \emph{spine} in the tropical base $B$ consists of a \Zaffine tree $\Gamma$, two 1-valent vertices $(v_1,v_2)$ of $\Gamma$, and a \Zaffine immersion $h\colon\Gamma\to B$ satisfying the following conditions:
	\begin{enumerate}
		\item The image of $h$ does not contain the origin $O\in B$.
		\item The vertices $v_1$ and $v_2$ are the only 1-valent vertices of $\Gamma$.
		In particular, $\Gamma$ is homeomorphic to a closed interval.
		\item There is no edge $e$ of $\Gamma$ such that $h(e)$ is contained in a ray starting from $O$.
		\item For every 2-valent vertex $v$, the vector $-\sum_{e\ni v} w_v(e)$ is either zero or points towards the origin $O\in B$, where $w_v(e)$ denotes the weight vector defined in \cref{def:Zaffine_map} (see \cref{fig:spine}).
	\end{enumerate}
\end{defin}

\begin{figure}[!ht]
	\centering
	\setlength{\unitlength}{0.2\textwidth}
	\begin{picture} (1,1)
	\put(0,0){\includegraphics[width=\unitlength]{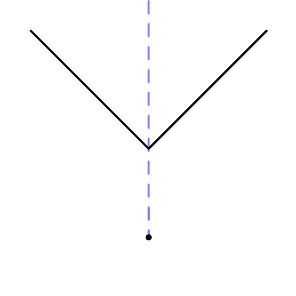}}
	\put(0.56,0.15){$O\in B$}
	\put(0.53,0.42){$v$}
	\put(-0.04,0.85){$v_1$}
	\put(0.94,0.85){$v_2$}
	\end{picture}
	\caption{}
	\label{fig:spine}
\end{figure}

\begin{defin} \label{def:extended_spine}
	An \emph{extended spine} in the tropical base $B$ consists of an unbounded \Zaffine tree $\Gamma$, two unbounded vertices $(v_1,v_2)$, and a \Zaffine immersion $h\colon\Gamma^\circ\to B$ such that the Conditions (1)-(4) of \cref{def:spine} hold.
\end{defin}

\begin{construction} \label{const:extension}
	Given any spine $L=[\Gamma,(v_1,v_2),h]$ in $B$, we can extend both its ends $v_1$ and $v_2$ straight with respect to the \Zaffine structure on $B\setminus O$.
	Since our Looijenga pair $(Y,D)$ is assumed to be positive, by \cite[Corollary 1.6]{Gross_Mirror_Log_2011}, this extension gives an extended spine in $B$, which is called the extended spine associated to $L$, and which we denote by $\hL=[\hGamma,(u_1,u_2),\hh\colon\hGamma^\circ\to B]$.
	This is the only place in the paper where we use the positivity assumption.
	
	Let $B^1$ denote the union of 1-dimensional cones in $B$ viewed as a simplicial cone complex.
	For any $b\in B^1\setminus O$, let $\be_b$ denote the primitive integral vector at $b$ in the direction of $\overrightarrow{Ob}$, and let $D_b$ denote the irreducible component of $D$ corresponding to the ray $\overrightarrow{Ob}$.
	Define
	\[\delta\coloneqq\sum_{x\in\hh\inv(B^1)\setminus\Gamma} \abs{d_x\hh\wedge\be_{\hat h(x)}} D_{\hat h(x)} \in\NE(Y),\]
	where $d_x$ denotes the derivative at $x$, and $\abs{\cdot}$ denotes the lattice length of the wedge product.
	We call $\delta$ the curve class associated to the extension from $L$ to $\hL$ (see \cite[\S 4]{Yu_Enumeration_cylinders_2015}).
\end{construction}

\begin{construction} \label{const:extended_spines}
	Fix a spine $L^0=[\Gamma^0,(v^0_1,v^0_2),h^0\colon\Gamma^0\to B]$ in $B$, which will play the role of an arbitrary spine in the paper.
	Let $L^1$, $L^2$ and $L^3$ be the spines in the statement of \cref{thm:gluing_formula}.
	For $i=0,\dots,3$, let
	\[\hL^i=[\hGamma^i,(u^i_1,u^i_2),\ \hh^i\colon(\hGamma^i)^\circ\to B]\]
	denote the extended spine in $B$ associated to $L^i$.
	
	Let us glue $\hL^1$ and $\hL^2$ along the vertices $v^1_2\in\hGamma^1$ and $v^2_1\in\hGamma^2$ to form an unbounded \Zaffine tree $\Gamma^\rd$ with four unbounded vertices $u^1_1,u^1_2,u^2_1,u^2_2\in\Gamma^\rd$.
		Let us denote the vertex $v^1_2=v^1_1\in \Gamma^\rd$ by $v^\rd$.
	The \Zaffine map $\hh^1\colon(\hGamma^1)^\circ\to B$ and $\hh^2\colon(\hGamma^2)^\circ\to B$ induce a \Zaffine map $h^\rd\colon(\Gamma^\rd)^\circ\to B$ (cf. \cref{fig:Ld}).
	We denote
	\[L^\rd\coloneqq[\Gamma^\rd,(u^1_1,u^1_2,u^2_1,u^2_2,v^\rd),h^\rd].\]
	
	\begin{figure}[!ht]
		\centering
		\setlength{\unitlength}{0.8\textwidth}
		\begin{picture} (1,0.273)
		\put(0,0){\includegraphics[width=\unitlength]{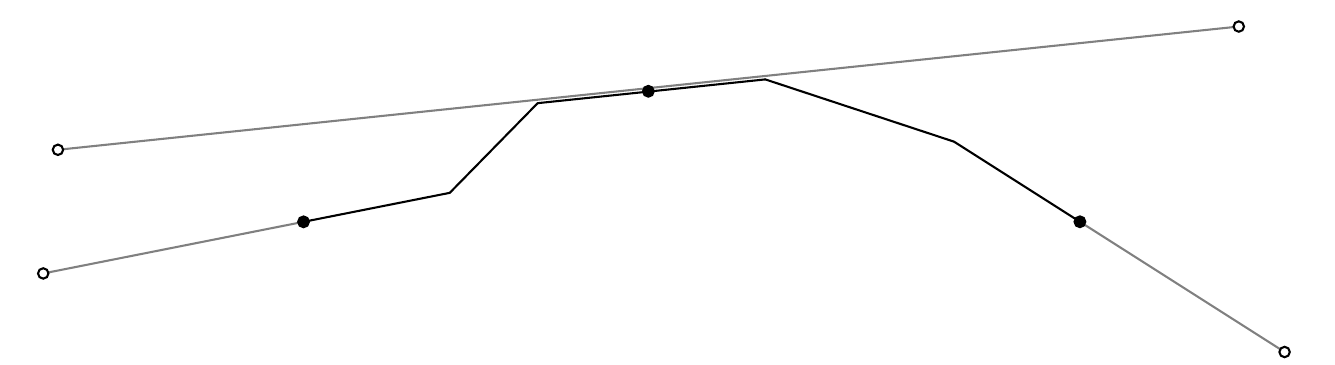}}
		\put(0.02,0.03){$u^1_1$}	
		\put(0.22,0.07){$v^1_1$}	
		\put(0.77,0.07){$v^2_2$}	
		\put(0.03,0.18){$u^2_1$}
		\put(0.94,0.26){$u^1_2$}	
		\put(0.47,0.22){$v^\rd$}	
		\put(0.975,0.00){$u^2_2$}
		\end{picture}
		\caption{}
		\label{fig:Ld}
	\end{figure}
	
	Let $\hGamma^4$ be the path connecting $u^2_1$ to $u^1_2$ in $\Gamma^\rd$ and let $\hh^4\colon(\hGamma^4)^\circ\to B$ the restriction of $h^\rd$.
	Denote $u^4_1\coloneqq u^2_1$, $u^4_2\coloneqq u^1_2$.
	We see that
	\[\hL^4=[\hGamma^4,(u^4_1,u^4_2),\ \hh^4\colon(\hGamma^4)^\circ\to B]\]
	is an extended spine in $B$ that is straight with respect to the \Zaffine structure on $B\setminus O$.
	
	For $i=0,\dots,4$, let $f^i_1$, $f^i_2$ denote the edges of $\hGamma^i$ connected to the unbounded vertices $u^i_1, u^i_2$ respectively.
	Let $u'^i_1$ denote the other endpoint of the edge $f^i_1$, and $u'^i_2$ the other endpoint of the edge $f^i_2$.
	Let $w^i_1\coloneqq w_{u'^i_1}(f^i_1)$ and $w^i_2\coloneqq w_{u'^i_2}(f^i_2)$ denote the weight vectors.	
\end{construction}

\begin{rem}
	Although the superscripts $0,\dots,4$ in the construction above are not very enlightening, they are practical when we apply the same arguments uniformly to all these spines.
\end{rem}

\begin{construction}\label{const:tB}
	Let $\tB\coloneqq\R\times B$.
	We endow the first factor $\R$ with the simplicial cone complex structure $\R=\R_{<0}\cup\{0\}\cup\R_{>0}$.
	Then $\tB$ obtains also the structure of a simplicial cone complex.
	Let $\tB'$ be a finite rational subdivision of $\tB$ into simplicial cones satisfying the following conditions:
	\begin{enumerate}
		\item Every simplicial cone in the subdivision is regular, i.e., the integer points in each simplicial cone can be generated by a subset of a basis of the underlying lattice.
		This condition is achievable by \cite[Chapter I Theorem 11]{Kempf_Toroidal_1973}.
		\item \label{item:tB:ray} For every $i=0,\dots,4$, there are two rays $\widetilde r^i_1$ and $\widetilde r^i_2$ in the subdivision $\tB'$ pointing to the directions of the vectors $(1, w^i_1)$ and $(1, w^i_2)$ respectively.
	\end{enumerate}
\end{construction}

\begin{construction}\label{const:CB}
	Let $C_B\coloneqq\R_{\ge 0}\times B$ and we embed $B$ into $C_B$ by the map $x\mapsto(1,x)$.
	Let $C'_B$ be a finite rational subdivision of $C_B$ into simplicial cones satisfying the following conditions:
	\begin{enumerate}
		\item Every simplicial cone in the subdivision is regular.
		\item Let $B'$ denote the subdivision of $B$ induced by $C'_B$.
		Let $B_\rf$ denote the union of the bounded faces of $B'$.
		The subscript $_\rf$ in the notation is short for ``finite''.
		Let $\partial B_\rf$ denote the boundary of $B_\rf$ viewed as a subset of $B$ and let $B_\rf^\circ$ denote the interior.
		We require that the origin $O$ and the images of the spines $L^0$, $L^1$, $L^2$ are all contained in the interior $B_\rf^\circ$.
		\item \label{item:CB:infinity} For every $i=0,\dots,4$, for every unbounded 2-dimensional cell $\rho$ in the subdivision $B'$ containing the ray $\hh^i(f^i_1\setminus u^i_1)$ (resp.\ $\hh^i(f^i_2\setminus u^i_2)$) except a finite part, we require that $\rho$ has two infinite 1-dimensional boundary faces pointing to the same direction.
		\item \label{item:CB:projection} Let $B'_{\infty}$ denote the conical subdivision of $B$ induced by the directions\footnote{Each infinite 1-dimensional cell in $B'$ has a direction vector which gives a ray in $B$ starting from $O$.} of the infinite 1-dimensional cells in $B'$.
		We require that the projection $\tB\to B$ induces a map of simplicial cone complexes $\tB'\to B'_{\infty}$.
	\end{enumerate}
\end{construction}

\begin{figure}[!ht]
	\centering
	\setlength{\unitlength}{0.7\textwidth}
	\begin{picture} (1,0.4545)
	\put(0,0){\includegraphics[width=\unitlength]{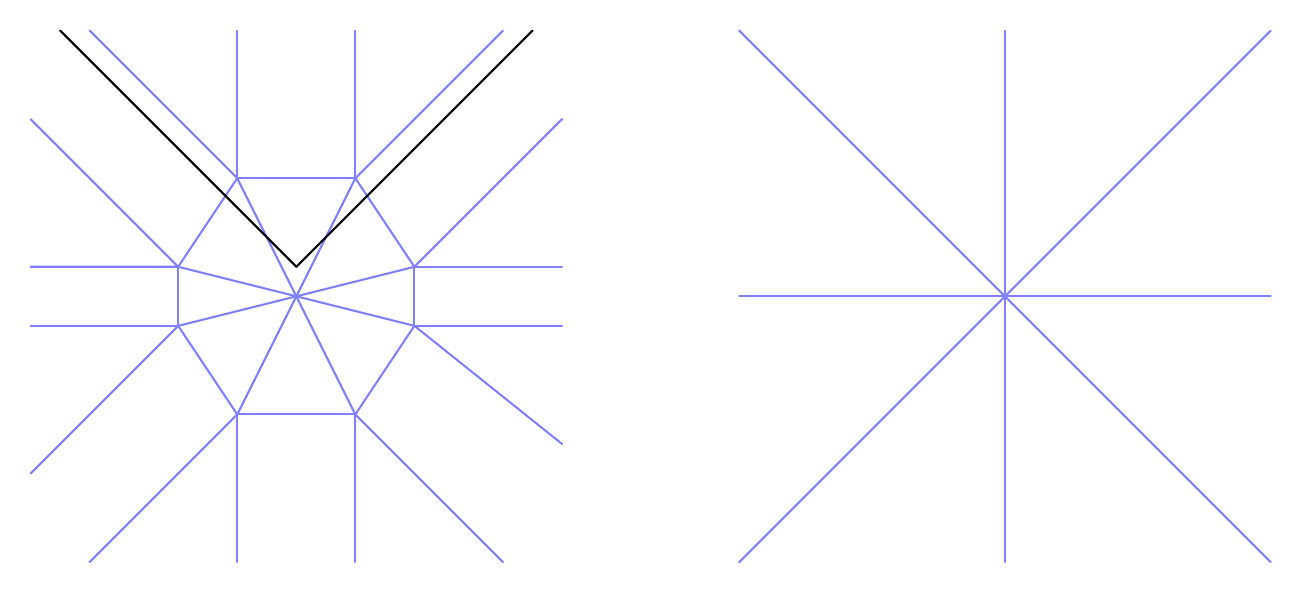}}
	\end{picture}
	\caption{An example of the subdivision $B'$ and the associated $B'_\infty$.}
	\label{fig:subdivision}
\end{figure}

\begin{construction}[see \cref{fig:graphs}]\label{const:graphs}
	We define two graphs with legs:
	\begin{itemize}
		\item The graph $G$ consists of three vertices $v_1, v_2, v_3$, an edge connecting $v_1, v_3$, an edge connecting $v_2, v_3$, two legs $l^2_1, l^1_2$ attached to the vertex $v_1$, two legs $l^1_1, l^2_2$ attached to the vertex $v_2$, and one leg $l_5$ attached to the vertex $v_3$.
		\item The graph $G'$ consists of three vertices $v_1, v_2, v_3$, an edge connecting $v_1, v_3$, an edge connecting $v_2, v_3$, two legs $l^1_1, l^1_2$ attached to the vertex $v_1$, two legs $l^2_1, l^2_2$ attached to the vertex $v_2$, and one leg $l_5$ attached to the vertex $v_3$.
	\end{itemize}
	Note that the only difference between $G$ and $G'$ is the labeling of the legs.
\end{construction}

\begin{figure}[!ht]
	\centering
	\setlength{\unitlength}{0.6\textwidth}
	\begin{picture} (1,0.426)
	\put(0,0){\includegraphics[width=\unitlength]{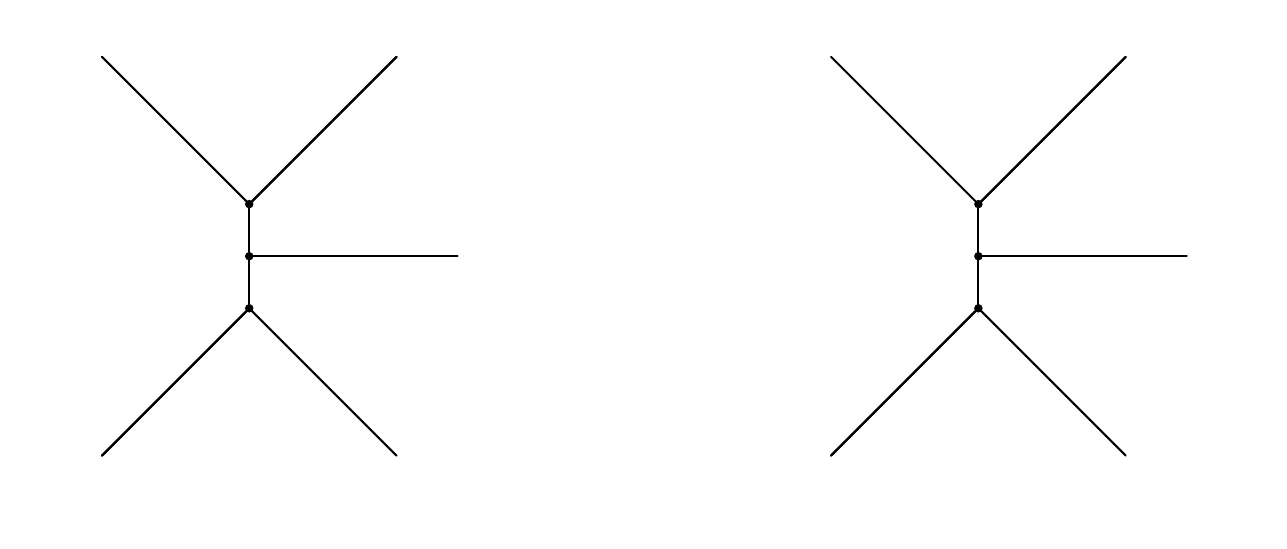}}

	\put(0.05,0.33){$l^2_1$}
	\put(0.05,0.09){$l^1_1$}
	\put(0.3,0.33){$l^1_2$}
	\put(0.3,0.09){$l^2_2$}
	\put(0.62,0.33){$l^1_1$}
	\put(0.62,0.09){$l^2_1$}
	\put(0.87,0.33){$l^1_2$}
	\put(0.87,0.09){$l^2_2$}
		
	\put(0.35,0.24){$l_5$}
	\put(0.92,0.24){$l_5$}
	
	\put(0.18,0.01){$G$}
	\put(0.74,0.01){$G'$}
		
	\end{picture}
	\caption{The graphs $G$ and $G'$.}
	\label{fig:graphs}
\end{figure}

\begin{construction} \label{const:Delta}
Following Abramovich-Caporaso-Payne \cite{Abramovich_Tropicalization_2012}, let $\oM\trop_{g,n}$ denote the coarse moduli space of extended tropical curves of genus $g$ with $n$ legs.
For the extended tropical curves parametrized by $\oM\trop_{0,5}$, we label the five legs by $l^1_1, l^1_2, l^2_1, l^2_2, l_5$ respectively.
For any $r\in(0,+\infty]$, let $G_r$ (resp.\ $G'_r$) denote the extended tropical curve such that
\begin{enumerate}
	\item the underlying graph is $G$ (resp.\ $G'$),
	\item the lengths of both vertical edges (as shown in \cref{fig:graphs}) are equal to $r$.
\end{enumerate}
Note that for $r=0$, we have the degenerate case $G_0=G'_0$.
Now varying $r$ from $0$ to $+\infty$, $G_r$ and $G'_r$ give two paths in $\oM\trop_{0,5}$.
We denote the union of the two paths by $\Delta\subset\oM\trop_{0,5}$.
\end{construction}

\begin{defin}\label{def:pointed_extended_tropical_cylinder}
	A \emph{pointed extended tropical cylinder} in the tropical base $B$ consists of an unbounded \Zaffine tree $\Gamma$ with exactly two unbounded vertices $(v_1,v_2)$, a marked point $v_3\in\Gamma$, and a \Zaffine immersion $h\colon\Gamma^\circ\to B$ satisfying the following condition:
	For every bounded vertex $v$ of $\Gamma$ such that $h(v)\neq O$, the \Zaffine map $h$ is balanced at $v$.
\end{defin}

\begin{defin}\label{def:tropical_double_cylinder}
	A \emph{pointed extended tropical double cylinder} in the tropical base $B$ consists of an unbounded \Zaffine tree $\Gamma$ with exactly four unbounded vertices $v_1,v_2,v_3,v_4$, a marked point $v_5\in\Gamma$, and a \Zaffine immersion $h\colon\Gamma^\circ\to B$ satisfying the same condition as in \cref{def:pointed_extended_tropical_cylinder}.
\end{defin}

\begin{construction}\label{const:subset_Ti}
	For $i=0,\dots,4$, let $T^i$ be the set of pointed extended tropical cylinders $(\Gamma,(v_1,v_2,v_3),h)$ satisfying the following conditions:
	\begin{enumerate}
		\item If $i=0$, then we have $h(v_3)=h^0(v^0_1)$.
		\item If $i=1,\dots,4$, then we have $h(v_3)=h^\rd(v^\rd)$. (Recall that $h^\rd(v^\rd)=h^1(v^1_2)=h^2(v^2_1)$.)
		\item Let $\Gamma^s$ denote the path in $\Gamma$ connecting the unbounded vertices $v_1$ and $v_2$.
		Then $(\Gamma^s,(v_1,v_2,v_3),h|_{\Gamma^s})$ is equal to  $(\hGamma^i,(u^i_1,u^i_2,v^i_1),\hh^i)$, where $v^3_1\coloneqq v^1_2=v^2_1\in\hGamma^3$ and $v^4_1\coloneqq v^\rd\in\hGamma^4$.
	\end{enumerate}
\end{construction}

\begin{construction}\label{const:subset_Td}
	Let $T^\rd$ be the set of pointed extended tropical double cylinders $(\Gamma,(v_1,v_2,v_3,v_4,v_5),h)$ satisfying the following conditions:
	\begin{enumerate}
		\item We have $h(v_5)=h^\rd(v^\rd)\in B$.
		\item Let $\Gamma^s$ denote the convex hull of the vertices $v_1,\dots,v_4$ in $\Gamma$.
		Then $(\Gamma^s,(v_1,\dots,\allowbreak v_5),\allowbreak h|_{\Gamma^s})$ is equal to $(\Gamma^\rd,(u^1_1,u^1_2,u^2_1,u^2_2,v^\rd),h^\rd)$.
	\end{enumerate}
\end{construction}

\begin{defin} \label{def:rational_tropical_curve_in_Bf}
	A \emph{pointed rational tropical curve in $B_\rf$ with $n$ boundary points} consists of a \Zaffine tree $\Gamma$, $(n+1)$ marked points $v_1,\dots,v_{n+1}\in\Gamma$, and a \Zaffine immersion $h\colon\Gamma\to B_\rf$ satisfying the following conditions:
	\begin{enumerate}
		\item \label{item:rational_tropical_curve_in_Bf:boundary1} For $i=1,\dots,n$, $h(v_i)$ lies in the boundary $\partial B_\rf$.
		\item \label{item:rational_tropical_curve_in_Bf:balancing1} For any vertex $v$ of $\Gamma$, if $h(v)$ lies in $B_\rf^\circ\setminus O$, then the \Zaffine map $h$ is balanced at $v$.
		\item \label{item:rational_tropical_curve_in_Bf:balancing2} For any vertex $v$ of $\Gamma$, if $h(v)$ lies in the relative interior of a 1-dimensional cell $\sigma$ of $\partial B_\rf$, let $\rho$ denote the unbounded 2-dimensional cell in $B'$ containing $\sigma$.
		Since the subdivision $B'$ of $B$ is induced by the simplicial conical subdivision $C'_B$ of $C_B$, the two other boundary faces of $\rho$ are two rays pointing to the same direction.
		Let $w_\rho$ denote this direction.
		Consider the quotient $T_{h(v)}B/w_\rho$, where $T_{h(v)}B$ denotes the tangent space of $B$ at $h(v)$.
		We require that $\sum_{e\ni v}w_v(e)$ is zero in $T_{h(v)}B/w_\rho$, where the sum is taken over all edges of $\Gamma$ connected to $v$.
		\item \label{item:rational_tropical_curve_in_Bf:boundary2} If $h\inv(\partial B_\rf)$ is a finite set, then $h\inv(\partial B_\rf)=\bigcup_{i=1}^n v_i$.
	\end{enumerate}
	A pointed rational tropical curve in $B_\rf$ with $n$ boundary points is called \emph{simple} if there is no balanced 2-valent vertex.
	Note that any pointed rational tropical curve in $B_\rf$ with $n$ boundary points can be made simple by removing all the balanced 2-valent vertices and gluing the corresponding edges.
\end{defin}

\begin{construction} \label{const:tropical_moduli}
Let $M_{n+1}(B_\rf)$ denote the space of simple pointed rational tropical curves in $B_\rf$ with $n$ boundary points.
It has a natural Hausdorff topological structure given by deformations of tropical curves (cf. \cite[Theorem 6.1]{Yu_Tropicalization_2014}).

For $i=0,\dots,4$, let $M^i_{2+1}(B_\rf)\subset M_{2+1}(B_\rf)$ denote the subset consisting of pointed rational tropical curves in $B_\rf$ with 2 boundary points $(\Gamma,(v_1,v_2,v_3),h\colon\Gamma\to B_\rf)$ such that $h(v_3)=h^0(v^0_1)$ for $i=0$, and $h(v_3)=h^\rd(v^\rd)$ for $i=1,\dots,4$.

Let $M^\rd_{4+1}(B_\rf)\subset M_{4+1}(B_\rf)$ denote the subset consisting of pointed rational tropical curves in $B_\rf$ with 4 boundary points $(\Gamma,(v_1,\dots,v_5),h\colon\Gamma\to B_\rf)$ satisfying the following conditions:
\begin{enumerate}
	\item We have $h(v_5)=h^\rd(v^\rd)\in B_\rf$.
	\item \label{item:tropical_moduli:intersection} Let $P_{13}$ denote the path in $\Gamma$ connecting $v_1$, $v_3$, and $P_{24}$ the path connecting $v_2$, $v_4$.
	Then the intersection $P_{13}\cap P_{24}$ is nonempty.
	Moreover, the vertex $v_5$ lies in the intersection.
\end{enumerate}
\end{construction}

\begin{construction} \label{const:TiTd}
	In the context of \cref{const:CB}, for every unbounded cell $\rho$ in the subdivision $B'$, we consider the linear projection from $\rho$ to $\rho\cap\partial B_\rf$.
	This gives a natural retraction map from $B$ to $B_\rf$.
	Under the retraction map, pointed extended tropical cylinders and pointed extended tropical double cylinders induce simple pointed rational tropical curves in $B_\rf$.
	So the sets $T^i$, for $i=0,\dots,4$, and $T^\rd$ induce subsets of $M(B_\rf)$, which we denote by $T^i_\rf$ and $T^\rd_\rf$ respectively.
	By construction, $T^i_\rf$ is a subset of $M^i_{2+1}(B_\rf)$, and $T^\rd_\rf$ is a subset of $M^\rd_{4+1}(B_\rf)$.
\end{construction}

\begin{prop}\label{prop:tropical_rigidity}
	For $i=0,\dots,4$, the subset $T^i_\rf\subset M^i_{2+1}(B_\rf)$ is a union of connected components of $M^i_{2+1}(B_\rf)$.
	Similarly, the subset $T^\rd_\rf\subset M^\rd_{4+1}(B_\rf)$ is a union of connected components of $M^\rd_{4+1}(B_\rf)$.
\end{prop}

\begin{lem} \label{lem:twig}
	Let $\Gamma$ be a \Zaffine tree which is not a point, and $r$ a vertex of $\Gamma$ called the \emph{root}.
	Let $h\colon\Gamma\to B$ be a \Zaffine immersion satisfying the following conditions:
	\begin{enumerate}
		\item The image $h(r)\neq O\in B$.
		\item \label{item:twig:balancing} For any vertex $v$ of $\Gamma$ except the root, if $h(v)$ lies in $B\setminus O$, then the \Zaffine map $h$ is balanced at $v$.
	\end{enumerate}
	Then there exists a ray $R$ in $B$ starting from $O\in B$ with rational slope such that the image $h(\Gamma)$ lies in $R$.
	In particular, the image $h(r)$ lies in $R$.
\end{lem}
\begin{proof}
	Let $V(\Gamma)$ denote the set of vertices of $\Gamma$.
	We call 1-valent vertices in $V(\Gamma)\setminus\{r\}$ leaves of $\Gamma$.
	By Condition (\ref{item:twig:balancing}), if $v$ is a leaf of $\Gamma$, then $h(v)=O\in B$.
	Therefore, for every edge $e$ of $\Gamma$ containing a leaf, the image $h(e)$ is contained in a ray $R_e$ in $B$ starting from $O\in B$.
	Since $h$ is a \Zaffine immersion, the ray $R_e$ has rational slope.
	Let $v$ be any vertex of $\Gamma$ except $r$ such that $h(v)\neq O$.
	Let $e_0,\dots,e_n$ be the edges connected to $v$.
	If we know that for $i=1,\dots,n$, the image $h(e_i)$ is contained in a ray $R_{e_i}$ in $B$ starting from $O$, then the rays $R_{e_1},\dots,R_{e_n}$ are all equal; moreover, by Condition (\ref{item:twig:balancing}), the image $h(e_0)$ is also contained in the same ray.
	Therefore, reasoning from the leaves to the root, we conclude that for all edge $e$ of $\Gamma$, the image $h(e)$ is contained in a same ray $R$ starting from $O\in B$ with rational slope.
\end{proof}

\begin{proof}[Proof of \cref{prop:tropical_rigidity}]
	The conditions in Constructions \ref{const:subset_Ti} and \ref{const:subset_Td} are all given by equalities, therefore, the subsets $T^i_\rf\subset M^i_{2+1}(B_\rf)$, $T^\rd_\rf\subset M^\rd_{4+1}(B_\rf)$ are closed.
	Now let us prove that they are open.
	
	Let $K=[\Gamma,(v_1,v_2,v_3),h\colon\Gamma\to B_\rf]\in T^i_\rf$, and let $\Gamma^s$ denote the path connecting $v_1$ and $v_2$ in $\Gamma$.
	We claim that the image $h(\Gamma^s)\subset B_\rf$ is fixed during small deformation of $K$ inside $M^i_{2+1}(B_\rf)$ (see \cref{fig:rigidity}).
	
	\begin{figure}[!ht]
		\centering
		\setlength{\unitlength}{0.3\textwidth}
		\begin{picture} (1,1)
		\put(0,0){\includegraphics[width=\unitlength]{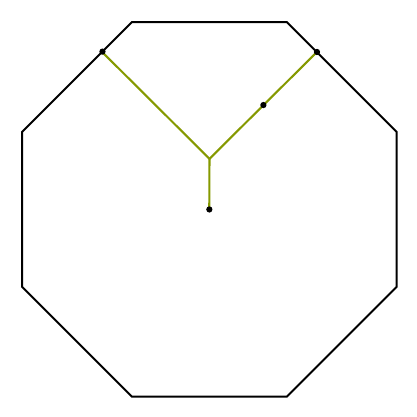}}
		\put(0.5,0.42){$O$}
		\put(0.16,0.9){$v_1$}	\put(0.76,0.9){$v_2$}
		\put(0.65,0.7){$v_3$}
		\put(0.8,0.05){$B_\rf$}
		\end{picture}
		\caption{An example of a tropical curve in $T^i_\rf$ having one bending vertex on the spine.}
		\label{fig:rigidity}
	\end{figure}

	Let $v$ be a 2-valent vertex of $\Gamma^s$ and let $e_1, e_2$ denote the two edges connected to $v$.
	We call $v$ a bending vertex of $\Gamma^s$ if we have $w_v(e_1)+w_v(e_2)\neq 0$ for the weight vectors.
	When we deform $K$ in $M^i_{2+1}(B_\rf)$, by the \Zaffine property, the slopes of the edges are fixed.
	Note that the image $h(v_3)\in B_\rf$ is always fixed by \cref{const:tropical_moduli}.
	Therefore, if $\Gamma^s$ contains no bending vertex, the image $h(\Gamma^s)\subset B_\rf$ must be fixed during small deformation of $K$.
		
	Now assume that $v$ is a bending vertex of $\Gamma^s$.
	Let $\Gamma_v$ be the closure of a connected component of $\Gamma\setminus\Gamma^s$ such that $\Gamma_v$ contains $v$.
	Applying \cref{lem:twig} to $\Gamma_v$, $v$ and $h|_{\Gamma_v}$, we see that $h(v)$ lies in a ray $R_v$ in $B$ starting from $O\in B$ with rational slope.
	Therefore, during any small deformation of $K$ inside $M^i_{2+1}(B_\rf)$, the image $h(v)$ must stay on the ray $R_v$.
	Since the image $h(v_3)$ and the slopes of the edges of $\Gamma^s$ are all fixed during deformation, we conclude that the image $h(\Gamma^s)\subset B_\rf$ is fixed during small deformation of $K$.
	So the claim holds.
	
	By \cref{const:CB} Condition (\ref{item:CB:infinity}), the preimage $h\inv(\partial B_\rf)$ in $\Gamma$ is exactly $\{v_1,v_2\}$.
	Therefore, by \cref{def:rational_tropical_curve_in_Bf} Conditions (\ref{item:rational_tropical_curve_in_Bf:boundary1}) and (\ref{item:rational_tropical_curve_in_Bf:boundary2}), the preimage $h\inv(\partial B_\rf)$ stays $\{v_1,v_2\}$ during small deformation of $K$ inside $M^i_{2+1}(B_\rf)$.
	We conclude that $[\Gamma^s,(v_1,v_2,v_3),h|_{\Gamma^s}]$ is fixed during small deformation of $K$ inside $M^i_{2+1}(B_\rf)$.
	In other words, the subset $T^i_\rf\subset M^i_{2+1}(B_\rf)$ is open, hence it is a union of connected components.
	
	Next let us show that the subset $T^\rd_\rf\subset M^\rd_{4+1}(B_\rf)$ is open using the same idea.
	Let $K=[\Gamma,(v_1,\dots,v_5),h\colon\Gamma\to B_\rf]\in T^\rd_\rf$, and let $\Gamma^s$ denote the convex hull of $v_1,\dots,v_4$ in $\Gamma$.
	Let $P_{13}$ be the path connecting $v_1$, $v_3$ in $\Gamma$, and $P_{24}$ the path connecting $v_2$, $v_4$ in $\Gamma$.
	We have $P_{13}\cap P_{24}=v_5$.
	By \cref{const:tropical_moduli} Condition (\ref{item:tropical_moduli:intersection}), during deformation of $K$ inside $M^\rd_{4+1}(B_\rf)$, the intersection $P_{13}\cap P_{24}$ stays nonempty.
	Therefore, by the balancing condition, the 4-valent vertex of $\Gamma^s$ cannot deform into two 3-valent vertices.
	So the intersection $P_{13}\cap P_{24}$ is always a point during deformation, which is always the vertex $v_5$.
	
	Now let $P_{14}$ be the path connecting $v_1$, $v_4$ in $\Gamma$, and $P_{23}$ the path connecting $v_2$, $v_3$ in $\Gamma$.
	The two paths intersect at $v_5$.
	By \cref{const:tropical_moduli}, the image $h(v_5)\in B_\rf$ is fixed during deformation of $K$ inside $M^\rd_{4+1}(B_\rf)$.
	Applying our reasoning for $T^i_\rf$ in the previous paragraphs to the paths $P_{14}$ and $P_{23}$, we deduce that $[\Gamma^s,(v_1,\dots,v_5),h|_{\Gamma^s}]$ is fixed during small deformation of $K$ inside $M^\rd_{4+1}(B_\rf)$.
	We conclude that the subset $T^\rd_\rf\subset M^\rd_{4+1}(B_\rf)$ is open, hence it is a union of connected components.
\end{proof}

\section{Analytic constructions}\label{sec:analytic_constructions}

In this section, we apply the tropical constructions in \cref{sec:tropical_constructions} to the constructions of analytic moduli spaces.
The main results are the properness statements in Propositions \ref{prop:properness} and \ref{prop:properness_d}.

We begin by replacing our target spaces by toric blowups for two technical reasons:
first, we want our curves to touch only the smooth locus of the boundary divisor; the second reason is related to the approximation $B_\rf\subset B$ explained in the beginning of \cref{sec:tropical_constructions}.
The toric blowups cause a bit of notation complications regarding curve classes, which are dealt with in Constructions \ref{const:sbeta} - \ref{const:sdbeta}.
First time readers are advised to skip these technical considerations.

We endow the projective line $\bbP^1_k$ with the standard toric structure.
The subdivision $\tB'$ in \cref{const:tB} induces a toric blowup\footnote{We refer to \cite[Chapter II]{Kempf_Toroidal_1973} for the relationship between subdivisions and blowups.} $\tY\to\bbP^1_k\times Y_k$.
The subdivision $B'_\infty$ in \cref{const:CB} induces a toric blowup $Y'\to Y_k$.
Let $\partial\tY$ and $\partial Y'$ denote the toroidal boundaries of $\tY$ and $Y'$ respectively.
By \cref{const:CB} Condition (\ref{item:CB:projection}), the map of simplicial cone complexes $\tB'\to B'_\infty$ induces a map of $k$-varieties $p_{Y,2}\colon\tY\to Y'$.

For $i=0,\dots,4$, let $m^i_1$ be the lattice length of the weight vector $w^i_1$ defined before \cref{const:tB}.
Let $\widetilde r^i_1$ be as in \cref{const:tB}(\ref{item:tB:ray}).
Let $r^i_1$ be the image of $\widetilde r^i_1$ under the map $\tB'\to B'_\infty$.
Let $\tD^i_1$ be the divisor in $\tY$ corresponding to the ray $\widetilde r^i_1$.
Let $D^i_1$ be the divisor in $Y'$ corresponding to the ray $r^i_1$.
Similarly, we define $m^i_2$, $\tD^i_2$ and $D^i_2$.

Now, let $i\in\set{0,\dots,4}$, and let $\beta$ be an element in $\NE(Y)$, the cone of curves of $Y$.

\begin{construction}\label{const:sbeta}
	 Let $s(\hL^i,\beta)$ denote the subset of $\NE(Y')$ consisting of elements $\gamma\in\NE(Y')$ satisfying the following conditions:
	 \begin{enumerate}
		\item The element $\gamma$ admits a representative which is a closed rational curve $C$ contained in $(Y'\setminus\partial Y')\cup(D^i_1\cup D^i_2)$ that intersects the divisor $D^i_1$ at exactly one point with multiplicity $m^i_1$, and the divisor $D^i_2$ at exactly one point with multiplicity $m^i_2$.
		\item Under the composite morphism $\NE(Y')\to\NE(Y_k)\xrightarrow{\sim}\NE(Y)$, the image of $\gamma$ equals $\beta\in\NE(Y)$.
	\end{enumerate}
\end{construction}

\begin{construction}\label{const:spbeta}
	Let $s^\rp(\hL^i,\beta)$ denote the subset of $\NE(\tY)$ consisting of elements $\gamma\in\NE(\tY)$ satisfying the following conditions:
	\begin{enumerate}
		\item \label{const:spbeta:intersect} The element $\gamma$ admits a representative which is a closed rational curve $C$ contained in $(\tY\setminus\partial\tY)\cup\tD^i_1\cup\tD^i_2$ that intersects the divisor $\tD^i_1$ transversally at exactly one point, and the divisor $\tD^i_2$ also transversally at exactly one point.
		\item \label{const:spbeta:P1} Under the projection $\NE(\tY)\to\NE(\bbP^1_k)$, the image of $\gamma$ is the fundamental class of $\bbP^1_k$.
		\item Under the composite morphism $\NE(\tY)\to\NE(Y_k)\xrightarrow{\sim}\NE(Y)$, the image of $\gamma$ equals $\beta\in\NE(Y)$.
	\end{enumerate}
	The superscript $^\rp$ in the notation is short for ``parametrized''.
\end{construction}

\begin{construction}\label{const:sdbeta}
	Let $s^\rd(\beta)$ denote the subset of $\NE(Y')$ consisting of elements $\gamma\in\NE(Y')$ satisfying the following conditions:
	\begin{enumerate}
		\item The element $\gamma$ admits a representative which is a closed rational curve $C$ contained in $(Y'\setminus\partial Y')\cup (D^1_1\cup D^1_2\cup D^2_1\cup D^2_2)$ that intersects the divisor $D^1_1$ (resp.\ $D^1_2, D^2_1, D^2_2$) at exactly one point with multiplicity $m^1_1$ (resp.\ $m^1_2, m^2_1, m^2_2$).
		\item Under the composite morphism $\NE(Y')\to\NE(Y_k)\xrightarrow{\sim}\NE(Y)$, the image of $\gamma$ equals $\beta\in\NE(Y)$.
	\end{enumerate}
\end{construction}

\begin{rem}
	Each of the subsets we defined in Constructions \ref{const:sbeta} - \ref{const:sdbeta} contains at most one element.
	To see this for $s(\hL^i,\beta)$, note that $N^1(Y')$ is a direct sum of $N^1(Y_k)$ with classes of exceptional curves, where $N^1(-)$ denotes divisors modulo numerical equivalence.
	So the conditions in \cref{const:sbeta} determine the curve class uniquely.
		The same reasoning applies to Constructions \ref{const:spbeta} and \ref{const:sdbeta} too.
\end{rem}

Recall that a stable map into a variety is a map from a pointed proper nodal curve to the variety having only finitely many automorphisms.
An equivalent characterization of having finitely many automorphisms is that every irreducible component of the domain nodal curve of genus 0 (resp.\ 1) which maps to a point must have at least 3 (resp.\ 1) special points on its normalization, where by special point, we mean the preimage of either a marked point or a node by the normalization map.
We refer to \cite{Kontsevich_Enumeration_1995,Fulton_Stable_1997} for the theory of stable maps in algebraic geometry, and to \cite{Yu_Gromov_compactness} for the theory in non-archimedean analytic geometry.

\begin{construction}\label{const:Mbeta}
	Let $\cM(\hL^i,\beta)$ denote the stack of 3-pointed rational stable maps into $Y'$ satisfying the following conditions:
	\begin{enumerate}
		\item The first marked point maps to $D^i_1$ with multiplicity greater than or equal to $m^i_1$; the second marked point maps to $D^i_2$ with multiplicity greater than or equal to $m^i_2$.
		(By ``multiplicity greater than or equal to'', we also allow the irreducible component containing the marked point to lie completely inside the divisor.
		So this condition is a Zariski closed condition.)
		\item The image of the fundamental class of the domain curve belongs to $s(\hL^i,\beta)$.
	\end{enumerate}
\end{construction}

\begin{construction}\label{const:Mpbeta}
	Let $\cM^\rp(\hL^i,\beta)$ denote the stack of 3-pointed rational stable maps into $\tY$ satisfying the following conditions:
	\begin{enumerate}
		\item The first marked point maps to $\tD^i_1$; the second marked point maps to $\tD^i_2$.
		\item The image of the fundamental class of the domain curve belongs to $s^\rp(\hL^i,\beta)$.
	\end{enumerate}
\end{construction}

The above parametrized version of moduli spaces is used in the construction of the counts $N(L,\beta)$ reviewed in \cref{sec:review_of_cylinder_counts}, it is also essential for \cref{prop:straight_spine}.

\begin{construction}\label{const:Mdbeta}
	Let $\cM^\rd(\beta)$ denote the stack of 5-pointed rational stable maps into $Y'$ satisfying the following conditions:
	\begin{enumerate}
		\item The first four marked points map respectively to the divisors $D^1_1, D^1_2, D^2_1, D^2_2$ with multiplicities greater than or equal to $m^1_1, m^1_2, m^2_1, m^2_2$ respectively.
		\item The image of the fundamental class of the domain curve belongs to $s^\rd(\beta)$.
	\end{enumerate}
\end{construction}

\begin{rem}
	Let $\cM(\hL^i,\beta)\an$ denote the analytification of the algebraic stack $\cM(\hL^i,\beta)$ (see \cite[\S 6.1]{Porta_Yu_Higher_analytic_stacks_2014}).
	By \cite[Theorem 8.7]{Yu_Gromov_compactness}, the \kanal stack $\cM(\hL^i,\beta)\an$ parametrizes 3-pointed rational analytic stable maps into the analytification $(Y')\an$ satisfying the conditions of \cref{const:Mbeta}.
	Similar remarks apply also to the stacks in Constructions \ref{const:Mpbeta} and \ref{const:Mdbeta}.
\end{rem}

The subdivision $B'$ of $B$ gives rise to an snc log-model $(\sX,H)$ of the $k$-variety $X_k$ (see \cite[\S 2]{Yu_Enumeration_cylinders_2015}).
Note that the generic fiber of $\sX$ is isomorphic to the $k$-variety $Y'$.
In other words, $\sX$ is an snc model of $Y'$.
Up to a finite ground field extension, we can assume furthermore that $\sX$ is a strictly semi-stable model of $Y'$ (see \cite{Kempf_Toroidal_1973}).
By construction, the Clemens polytope associated to $\sX$ is isomorphic to $B_\rf$ (i.e.\ the finite part of $B'$) as a simplicial complex.
Let
\[\tau_\rf\colon (Y')\an\to B_\rf\]
denote the corresponding retraction map (see \cite[Proposition 2.7]{Yu_Gromov_compactness}).
Note that the toroidal boundary of $(Y')\an$ maps onto $\partial B_\rf$.

The following lemma relates analytic stable maps in $(Y')\an$ to tropical curves in $B_\rf$.

\begin{lem} \label{lem:tropicalization}
	Let $[C, (s_1,s_2,s_3), f\colon C\to (Y')\an]$ be an analytic stable map in $\cM(\hL^i,\allowbreak\beta)\an$.
	The composite map $C\xrightarrow{f}(Y')\an\xrightarrow{\tau_\rf} B_\rf$ factors as $C\xrightarrow{c}\Gamma\xrightarrow{h}B_\rf$, where $C\xrightarrow{c}\Gamma$ contracts all paths in $C$ that are contracted by $\tau_\rf\circ f$.
	Let $v_i\coloneqq c(s_i)$ for $i=1,2,3$.
	Then $[\Gamma,(v_1,v_2,v_3),h]$ is a pointed rational tropical curve in $B_\rf$ with 2 boundary points in the sense of \cref{def:rational_tropical_curve_in_Bf}.
\end{lem}
\begin{proof}
Let $\fX$ be the completion of $\sX$ along its special fiber.
By \cite[Theorem 1.5]{Yu_Gromov_compactness}, up to passing to a finite ground field extension, we can find a strictly semi-stable formal model $\fC$ of $C$ and a morphism of formal schemes $\ff\colon\fC\to\fX$ such that $\ff_\eta\simeq f$.
Let $S_\fC$ be the Clemens polytope of $\fC$ and $\tau_\fC\colon C\to S_\fC$ the retraction map.
By the functoriality of Clemens polytope (see \cite[\S 2]{Yu_Gromov_compactness}), we obtain a commutative diagram
\[\begin{tikzcd}
C \ar{r}{f} \ar{d}{\tau_\fC} & (Y')\an \ar{d}{\tau_\rf} \\
S_\fC \ar{r}{S_f} & B_\rf,
\end{tikzcd}\]
where $S_f$ is affine on every edge of $S_\fC$.
The map $S_f\colon S_\fC\to B_\rf$ factors as $S_\fC\xrightarrow{c'}\oS_\fC\xrightarrow{\oS_f}B_\rf$ where $c'$ contracts all edges of $S_\fC$ that are contracted by $S_f$, so that $\oS_f$ is an immersion.
Thus the map $\Gamma\xrightarrow{h} B_\rf$ is isomorphic to $\oS_\fC\xrightarrow{\oS_f}B_\rf$.
This shows that $\Gamma$ is a finite \Zaffine tree, and $h\colon\Gamma\to B_\rf$ is a \Zaffine immersion.

Now let us verify Conditions (1)-(4) in \cref{def:rational_tropical_curve_in_Bf}.
Since $f(s_1),f(s_2)\in(\partial Y')\an$ and $\tau_\rf((\partial Y')\an)=\partial B_\rf\subset B_\rf$, Condition (\ref{item:rational_tropical_curve_in_Bf:boundary1}) is satisfied.
Next, note that any point $p\in B^\circ_\rf\setminus O$ has an open neighborhood $V$ in $B_\rf$ such that the map $\tau\inv_\rf(V)\to V$ is a trivial affinoid torus fibration.
So Condition (\ref{item:rational_tropical_curve_in_Bf:balancing1}) is satisfied (see \cite[Theorem 6.14]{Baker_Nonarchimedean_2011}).
For Condition (\ref{item:rational_tropical_curve_in_Bf:balancing2}), consider the extension of the valuation map $\val\colon\Gmk\an\to\R$ to the map $\widebar{\val}\colon(\bbA^1_k)\an\to\R\cup\{+\infty\}$.
Let $\phi\colon\R\cup\{+\infty\}\to(-\infty,0]$ be the map contracting $[0,+\infty]$ to the point 0.
Let $\psi\coloneqq\phi\circ\widebar{\val}\colon(\bbA^1_k)\an\to(-\infty,0]$ be the composition.
For any point $p$ in the relative interior of a 1-dimensional cell of $\partial B$, there exists a small open neighborhood $V$ of $p$ in $B_\rf$ such that the map $\tau\inv_\rf(V)\to V$ is isomorphic to the map $(\psi,\val)\colon(\bbA^1_k)\an\times\Gmk\an\to(-\infty,0]\times\R$ localized over an open subset of $(-\infty,0]\times\R$.
Therefore, by projecting $\tau_\rf\inv(V)$ to the factor $\Gmk\an$, \cite[Theorem 6.14]{Baker_Nonarchimedean_2011} implies Condition (\ref{item:rational_tropical_curve_in_Bf:balancing2}) too.
Finally, if $h\inv(\partial B_\rf)$ is a finite set, then $f(C)$ does not contain any component of $(\partial Y')\an$.
So by Constructions \ref{const:sbeta} and \ref{const:Mbeta}, $f(C)$ cannot have extra intersections with $(\partial Y')\an$ except at $s_1$ and $s_2$.
Thus the following lemma implies Condition (\ref{item:rational_tropical_curve_in_Bf:boundary2}).
\end{proof}

\begin{rem}
	We mention a generalization \cite[Theorem 4.1]{Yu_Tropicalization_2014} and a variant \cite[Theorem 3.4.2]{Ranganathan_Skeletons_I_2015} of \cref{lem:tropicalization}.
\end{rem}

\begin{lem} \label{lem:boundary}
	In the context of \cref{lem:tropicalization}, let $v$ be a vertex of $\Gamma$ and $U$ a small closed connected neighborhood of $v$ in $\Gamma$.
	Assume that $h(v)\in\partial B_\rf$, $h(U\setminus v)\cap\partial B_\rf=\emptyset$, and $O\not\in h(U)$.
	Then there is a point $x\in C$ such that $c(x)=v$ and $f(x)\in(\partial Y')\an$.
\end{lem}
\begin{proof}
	Suppose to the contrary that such a point $x$ does not exist.
	Since $h(U\setminus v)\cap\partial B_\rf=\emptyset$, then $f(c\inv(U))\cap(\partial Y')\an=\emptyset$.
	Using formal models as in the proof of \cref{lem:tropicalization}, we see that the composite map $c\inv(U)\xrightarrow{f} X_k\an\xrightarrow{\tau}B$ factors as $c\inv(U)\xrightarrow{c'}\Delta\xrightarrow{h'}B$ where $c\inv(U)\xrightarrow{c'}\Delta$ contracts all paths in $c\inv(U)$ that are contracted by $(\tau\circ f)$.
	Then the composite map $\Delta\xrightarrow{h'}B\to B_\rf$ factors as $\Delta\xrightarrow{r} U\xrightarrow{h} B_\rf$, where $r$ contracts a subtree of $\Delta$ to the point $v\in U$.
	Note that we have a natural embedding $U\subset\Delta$ and $r\colon\Delta\to U$ is a retraction map.
	Since $h'$ is balanced at $v$, there is at least one edge $e$ of $\Delta$ connected to $v$ such that $r(e)=v$.
	Let $\Delta_e$ be the closure of the connected component of $\Delta\setminus v$ containing $e$.
	Let $v'$ be a 1-valent vertex of $\Delta_e$ other than $v$.
	Since $O\not\in h(U)$, we have $h'(v')\neq O$.
	Then the \Zaffine map $h'$ must be balanced at the vertex $v'$, which contradicts the fact that $v'$ is 1-valent.
	So the proof is complete.
\end{proof}

By \cref{lem:tropicalization}, we obtain a tropicalization map
\[\tau^i_\cM\colon\cM(\hL^i,\beta)\an\to M_{2+1}(B_\rf).\]
Similarly, we have a tropicalization map for the stack $\cM^\rd(\beta)$
\[\tau_{\cM^\rd}\colon\cM^\rd(\beta)\an\to M_{4+1}(B_\rf).\]

A point in the stack $\cM^\rp(\hL^i,\beta)\an$ corresponds to a 3-pointed rational analytic stable map into $\tY\an$.
If we compose this stable map with the map $p\an_{Y,2}\colon\tY\an\to(Y')\an$, we obtain a map from a nodal curve to $(Y')\an$.
Its tropicalization is a pointed rational tropical curve in $B_\rf$ with 2 boundary points in the sense of \cref{def:rational_tropical_curve_in_Bf}.
So we obtain a tropicalization map
\[\tau^i_{\cM^\rp}\colon\cM^\rp(\hL^i,\beta)\an\to M_{2+1}(B_\rf).\]

Let us denote
\begin{align*}
\cM(\hL^i,\beta,T)&\coloneqq(\tau^i_\cM)\inv(T^i_\rf),\\
\cM^\rp(\hL^i,\beta,T)&\coloneqq(\tau^i_{\cM^\rp})\inv(T^i_\rf),\\
\cM^\rd(\beta,T)&\coloneqq(\tau_{\cM^\rd})\inv(T^\rd_\rf),
\end{align*}
where the preimages are understood merely as functors of points (without extra geometric structures a priori).

Let
\begin{align*}
	W^0&\coloneqq\tau_\rf\inv(h^0(v^0_1))\subset (Y')\an,\\
	W^\rp&\coloneqq\Gmk\an\times W^0\subset \tY\an,\\
	W^\rd&\coloneqq\tau\inv_\rf(h^\rd(v^\rd))\subset (Y')\an.
\end{align*}

\begin{prop}\label{prop:properness}
	Consider the evaluation map of the third marked point for the stack $\cM(\hL^0,\beta)\an$
	\[\ev_3\colon\cM(\hL^0,\beta)\an\to(Y')\an.\]
	Let $\cM(\hL^0,\beta)_W\coloneqq\ev\inv_3 W^0$ and
	\[\cM(\hL^0,\beta,T)_W\coloneqq \cM(\hL^0,\beta)_W\cap\cM(\hL^0,\beta,T).\]
	Then $\cM(\hL^0,\beta,T)_W$ is a \kanal stack, and the restriction
	\[\ev_3|_{\cM(\hL^0,\beta,T)_W}\colon \cM(\hL^0,\beta,T)_W \to W^0\]
	is a proper map.
	Similarly, consider the evaluation map of the third marked point for the stack $\cM^\rp(\hL^0,\beta)\an$
	\[\ev_3\colon\cM^\rp(\hL^0,\beta)\an\to\tY\an.\]
	Let $\cM^\rp(\hL^0,\beta)_W\coloneqq\ev_3\inv W^\rp$ and
	\[\cM^\rp(\hL^0,\beta,T)_W\coloneqq \cM^\rp(\hL^0,\beta)_W\cap\cM^\rp(\hL^0,\beta,T).\]
	Then $\cM^\rp(\hL^0,\beta,T)_W$ is a \kanal stack, and the restriction
	\[\ev_3|_{\cM^\rp(\hL^0,\beta,T)_W}\colon \cM^\rp(\hL^0,\beta,T)_W \to W^\rp\]
	is a proper map.
\end{prop}
\begin{proof}
	Let us prove the first statement.
	The same argument applies to the second statement.
	Consider the tropicalization map
	\[\tau^0_\cM\colon\cM(\hL^0,\beta)\an\to M_{2+1}(B_\rf).\]
	Let $[C,(p_1,p_2,p_3),f\colon C\to (Y')\an]$ be the analytic stable map corresponding to a point in $\cM(\hL^0,\beta)_W$.
	Let $[\Gamma,(v_1,v_2,v_3),h\colon\Gamma\to B_\rf]$ be the associated rational tropical curve in $B_\rf$ with two boundary points.
	Since $\cM(\hL^0,\beta)_W=\ev_3\inv(W^0)$ and $W^0=\tau_\rf\inv(h^0(v^0_1))$, we have $\tau_\rf(f(p_3))=h^0(v^0_1)$.
	So $h(v_3)=h^0(v^0_1)$.
	We deduce that
	\[\tau^0_\cM(\cM(\hL^0,\beta)_W)\subset M^0_{2+1}(B_\rf),\]
	where $M^0_{2+1}(B_\rf)$ is defined in \cref{const:tropical_moduli}.
	
	By \cref{prop:tropical_rigidity}, the subset $T^0_\rf\subset M^0_{2+1}(B_\rf)$ is a union of connected components.
	The tropicalization map $\tau^0_\cM$ is continuous by \cite[Theorem 8.1]{Yu_Tropicalization_2014}.
	Therefore, the substack $\cM(\hL^0,\beta,T)_W\subset\cM(\hL^0,\beta)_W$ is a union of connected components; in particular, it is a \kanal stack.
	Since the algebraic stack $\cM(\hL^0,\beta)$ is proper over $k$, by \cite[Proposition 6.4]{Porta_Yu_Higher_analytic_stacks_2014}, its analytification $\cM(\hL^0,\beta)\an$ is a proper \kanal stack.
	So the map
	\[\ev_3\colon\cM(\hL^0,\beta)\an\to(Y')\an\]
	is proper.
	By base change, the map
	\[\ev_3|_{\cM(\hL^0,\beta)_W}\colon\cM(\hL^0,\beta)_W\to W^0\]
	is proper.
	Since the substack $\cM(\hL^0,\beta,T)_W\subset\cM(\hL^0,\beta)_W$ is a union of connected components, we deduce that the map
	\[\ev_3|_{\cM(\hL^0,\beta,T)_W}\colon \cM(\hL^0,\beta,T)_W \to W^0\]
	is proper, completing the proof.
\end{proof}

\begin{lem} \label{lem:components_of_the_domain_curve}
	Any analytic stable map $[C, (\text{marked points}),f\colon C\to(Y')\an]$ in $\cM(\hL^0,\beta,T)$ (resp.\ $\cM^\rd(\beta,T)$) satisfies the following properties:
	\begin{enumerate}
		\item \label{item:components:boundary} There is no irreducible component of $C$ mapping into the boundary $(\partial Y')\an$.
		\item \label{item:components:marked_points} The preimage $f\inv((\partial Y')\an)$ contains exactly the first two (resp.\ four) marked points on $C$.
		\item \label{item:components:unstable} Let $C'$ be an irreducible component of $C$ that contains at most two special points; then $f(C')$ does not meet the boundary $(\partial Y')\an$ and $(\tau\circ f)(C')=O\in B$.
		\item \label{item:components:stabilization} Let $C_\mathrm{m}\subset C$ be the union of irreducible components containing the marked points.
		Let $\st(C)$ be the stabilization of the pointed domain curve (obtained by iterated contractions of non-stable components).
		We have $C_\mathrm{m}=\st(C)$.
	\end{enumerate} 
\end{lem}
\begin{proof}
	Let us prove the lemma in the case of $\cM^\rd(\beta,T)$.
	The proof in the case of $\cM(\hL^0,\beta,T)$ is similar and simpler.
	
	Let $Z=(\Gamma,(v_1,\dots,v_5),h\colon\Gamma\to B_\rf)$ denote the pointed rational tropical curve in $B_\rf$ with four boundary points associated to the stable map.
	Let $C_\rb$ denote the union of the irreducible components of $C$ mapping into the boundary $(\partial Y')\an$.
	Let $C'_\rb$ be a connected component of $C_\rb$.
	Since $Z$ belongs to $T^\rd_\rf$, the intersection $h(\Gamma)\cap\partial B_\rf$ consists of finitely many points in $\partial B_\rf$.
	So $(\tau_\rf\circ f)(C'_\rb)$ is a point in $\partial B_\rf$, which we denote by $b$.
	Since $\tau_\rf\inv(b)$ is an affinoid domain in $(Y')\an$, the image $f(C'_\rb)\subset(Y')\an$ must be a point in $(\partial Y')\an$.
	
	Let $C_\ri$ be the union of the irreducible components of $C$ not contained in $C_\rb$.
	We claim that $C'_\rb\cap C_\ri$ contains exactly one point.
	To see this, let $v\in\set{v_1,\dots,v_4}$ be the vertex of $\Gamma$ corresponding to $C'_\rb$; let $x\in C'_\rb\cap C_\ri$ be an intersection point; and let $K$ be the irreducible component of $C_\ri$ containing $x$.
	If $(\tau_\rf\circ f)(K)\subset\partial B_\rf$, using the fact that $h(\Gamma)\cap\partial B_\rf$ is a finite set, and that the fibers of $\tau_\rf\colon (Y')\an\to B_\rf$ over $\partial B_\rf$ are affinoid, we see that $f(K)$ is a point in $(Y')\an$.
	Since $K\subset C_\ri$, $f(K)$ is a point in $(Y'\setminus\partial Y')\an$.
	Then $C'_\rb$ and $K$ cannot intersect.
	Therefore, $(\tau_\rf\circ f)(K)$ cannot be contained in $\partial B_\rf$.
	So the intersection point $x=C'_\rb\cap K$ gives rise to an edge of $\Gamma$ connected to $v$.
	Since $Z$ belongs to $T^\rd_\rf$, $v$ is a 1-valent vertex.
	Since each intersection point of $C'_\rb\cap C_\ri$ gives rise to a different edge of $\Gamma$ connected to $v$, we conclude that $C'_\rb\cap C_\ri$ contains exactly one point.
	
	Moreover, $Z\in T^\rd_\rf$ implies that $v_1,\dots,v_4$ are four different points on $\Gamma$.
	So $C'_\rb$ contains at most one marked point.
	Let $C_1$ be the union of the irreducible components of $C$ not contained in $C'_\rb$.
	We have $C_1=(C_\rb\setminus C'_\rb)\cup C_\ri$.
	So
	\[C'_\rb\cap C_1=C'_\rb\cap\big((C_\rb\setminus C'_\rb)\cup C_\ri\big)=C'_\rb\cap C_\ri\]
	contains exactly one point.
	Recall that $C'_\rb$ has arithmetic genus zero and contains at most one marked point.
	So there exists an irreducible component $E$ of $C'_\rb$ which contains at most two special points.
	As $f$ is constant on $E$, we get a contradiction to the stability condition of the stable map $f$.
	Therefore, $C_\rb$ must be empty.
	This proves Property (\ref{item:components:boundary}).
	
	Using Property (\ref{item:components:boundary}) and the conditions on the intersections numbers in \cref{const:sbeta} (resp.\ \cref{const:sdbeta}), we obtain Property (\ref{item:components:marked_points}).
	
	Now let $C'$ be the irreducible component of $C$ in Property (\ref{item:components:unstable}).
	In this case, $C$ is reducible, so $C'$ contains at least one node of $C$.
	By Property (\ref{item:components:marked_points}), $f(C')$ can meet the boundary $(\partial Y')\an$ in at most one point.
	If $f(C')$ meets $(\partial Y')\an$ at exactly one point, then we claim that $(\tau_\rf\circ f)(C')$ must contain the intersection between $B_\rf$ and a rational ray starting from $O$ in $B$.
	
	In order to see the claim, note that the retraction map $\tau\colon X_k\an\to B$ extends to a retraction map $\otau\colon(Y'_k)\an\to\oB'$, where $\oB'$ is the closure of $B'$ inside $(Y'_k)\an$ (see \cite[Theorem 4.13]{Gubler_Skeletons_2014}).
	The composite map $C\xrightarrow{f}(Y')\an\xrightarrow{\otau}\oB'$ factors as $C\xrightarrow{\oc}\oGamma\xrightarrow{\oh}\oB'$ where $C\xrightarrow{\oc}\oGamma$ contracts all paths in $C$ that are contracted by $\otau\circ f$.
	If $f(C')$ meets $(\partial Y')\an$ at exactly one point, then there is exactly one point $r$ of $\oGamma$ such that $\oh(r)\in\partial\oB'$.
	Let $\oR$ denote the ray connecting $O$ and $\oh(r)$ in $\oB'$.
	Applying the same argument as in the proof of \cref{lem:twig} to $\oh\colon\oGamma\to\oB'$ with root $r$, we see that the image $\oh(\oGamma)$ is equal to the ray $\oR$.
	Then the image $(\tau_\rf\circ f)(C')\subset B_\rf$ contains the intersection $\oR\cap B_\rf$.
	So the claim holds.
	
	The claim contradicts the fact that the pointed rational tropical curve $Z$ belongs to $T^\rd_\rf$, because the image of any tropical curve in $T^\rd_\rf$ cannot contain such a ray.
	Therefore, $f(C')$ must lie in the interior $(Y'\setminus \partial Y')\an\simeq X\an_k$.
	In this case, by the balancing condition, $(\tau\circ f)(C')$ is a point in $B$.
	By the stability condition, $f|_{C'}$ is nontrivial.
	Moreover, since $\tau\colon X\an_k\to B$ is an affinoid torus fibration outside $O\in B$, the only possibility for the image $(\tau\circ f)(C')$ is the origin $O\in B$.
	This proves Property (\ref{item:components:unstable}).
	
	Let $C_\mathrm{m}$ and $\st(C)$ be as in Property (\ref{item:components:stabilization}).
	Let us first prove that $C_\mathrm{m}$ is connected.
	Assume on the contrary that $C_\mathrm{m}$ is not connected.
	Let $C_\rn\subset C$ be the union of irreducible components of $C$ not contained in $C_\mathrm{m}$.
	Let $C'_\rn$ be a connected component of $C_\rn$ that meets at least two connected components of $C_\mathrm{m}$.
	
	We claim that $(\tau_\rf\circ f)(C'_\rn)$ must be the point $h^\rd(v^\rd)\in B_\rf$.
	To see this, say that $C'_\rn$ meets two connected components $C^1_\mathrm{m}$ and $C^2_\mathrm{m}$ of $C_\mathrm{m}$.
	Let $\Gamma^1_\mathrm{m}$, $\Gamma^2_\mathrm{m}$ and $\Gamma'_\rn$ be the connected subtrees of $\Gamma$ associated to $C^1_\mathrm{m}$, $C^2_\mathrm{m}$ and $C'_\rn$ respectively.
	The restrictions of $h$ to these subtrees are balanced away from $\partial B_\rf\cup O$.
	Note that each of $\Gamma^1_\mathrm{m}$ and $\Gamma^2_\mathrm{m}$ contains some marked points.
	Since $Z$ belongs to $T^\rd_\rf$, it follows from the shape of $Z$ and the balancing condition that $\Gamma'_\rn$ must be the point $v_5$.
	Hence we obtain $(\tau_\rf\circ f)(C'_\rn)=h(\Gamma'_\rn)=h(v_5)=h^\rd(v^\rd)$.
		
	Since $\tau_\rf\inv(h^\rd(v^\rd))$ is affinoid, the image $f(C'_\rn)$ is a point in $X\an_k$.
	Then by the stability condition, the intersection $C'_\rn\cap C_\mathrm{m}$ has at least three points. 
	Since $C$ is rational, we see that $C_\mathrm{m}$ must have at least three connected components.
	Using again the fact that $Z$ belongs to $T^\rd_\rf$, we see that only the following situation is possible: $C_\mathrm{m}$ has exactly three connected components, one containing the last marked point, one containing two among the first four marked points, and one containing the remaining two marked points.
	We denote the three connected components by $C^1_\mathrm{m}$, $C^2_\mathrm{m}$, $C^3_\mathrm{m}$ respectively.
	Since $C$ is rational, $C'_\rn$ intersects each of $C^1_\mathrm{m}$, $C^2_\mathrm{m}$ and $C^3_\mathrm{m}$ exactly once.
	Since $C^1_\mathrm{m}$ has only one marked point, by Property (\ref{item:components:unstable}), $C^1_\mathrm{m}$ must contain two nodes of $C$.
	So there exists a connected component $C''_\rn$ of $C_\rn$ different from $C'_\rn$ that meets $C^1_\mathrm{m}$.
	By the same reasoning as for $C'_\rn$, the intersection $C''_\rn\cap C_\mathrm{m}$ has at least three points.
	However, since $C$ is rational and $C_\mathrm{m}\cup C'_\rn$ is connected, we see that $C''_\rn\cap(C_\mathrm{m}\cup C'_\rn)$ can contain at most one point.
	So we have reached a contradiction.
	Therefore, $C_\mathrm{m}$ is connected.
	
	By Property (\ref{item:components:unstable}), every irreducible component in $C_\mathrm{m}$ contains at least three special points.
	Combining with the fact that $C_\mathrm{m}$ is connected, we see that $C_\mathrm{m}$ is a stable pointed curve.
	Note that for any pointed rational nodal curve $C'\supset C_\mathrm{m}$ whose marked points are all contained in $C_\mathrm{m}$, there exists at least one irreducible component of $C'$ not contained in $C_\mathrm{m}$ which contains at most two nodes.
	Therefore, after iterated contractions of non-stable components of $C$, we will end up with $C_\mathrm{m}$.
	In other words, we have $C_\mathrm{m}=\st(C)$, completing the proof of Property (\ref{item:components:stabilization}).
\end{proof}

\begin{construction}\label{const:tDelta}
	Let $\bcM_{0,5}$ denote the moduli space of 5-pointed rational algebraic stable curves over $k$.
	Let $\tau_{\bcM_{0,5}}\colon\bcM\an_{0,5}\to\oM\trop_{0,5}$ denote the tropicalization map sending pointed analytic stable curves to the associated extended tropical curves (see \cite{Abramovich_Tropicalization_2012}).
	Let $\Delta\subset\oM\trop_{0,5}$ be as in \cref{const:Delta}.
	We set $\tDelta\coloneqq\tau_{\bcM_{0,5}}\inv(\Delta)\subset\bcM\an_{0,5}$.
\end{construction}

\begin{prop}\label{prop:properness_d}
Let
\[\Phi=(\mathrm{st},\ev_5)\colon\cM^\rd(\beta)\an\to\bcM\an_{0,5}\times (Y')\an\]
be the map given by stabilization of domain curves and the evaluation map of the fifth marked point.
Let
\[\cM^\rd(\beta)_W\coloneqq\Phi\inv(\tDelta\times W^\rd),\]
\[\cM^\rd(\beta,T)_W\coloneqq\cM^\rd(\beta)_W\cap\cM^\rd(\beta,T).\]
Then $\cM^\rd(\beta,T)_W$ is a \kanal stack, and the restriction
\[\Phi|_{\cM^\rd(\beta,T)_W} \colon \cM^\rd(\beta,T)_W \to\tDelta\times W^\rd\]
is a proper map.
\end{prop}
\begin{proof}
	Consider the tropicalization map
	\[\tau_{\cM^\rd}\colon\cM^\rd(\beta)\an\to M_{4+1}(B_\rf).\]
	Let $[C,(p_1,\dots,p_5),f\colon C\to (Y')\an]$ be the analytic stable map corresponding to a point in $\cM^\rd(\beta)_W$.
	Let $[\Gamma,(v_1,\dots,v_5),h\colon\Gamma\to B_\rf]$ be the associated pointed rational tropical curve in $B_\rf$ with four boundary points.
	Since $\cM^\rd(\beta)_W=\Phi\inv(\tDelta\times W^\rd)$ and $W^\rd=\tau_\rf\inv(h^\rd(v^\rd))$, we have $\tau_\rf(f(p_5))=h^\rd(v^\rd)$.
	So $h(v_5)=h^\rd(v^\rd)$.
	
	Let $\st(C)$ denote the stabilization of the 5-pointed curve $(C,(p_1,\dots,p_5))$.
	Since $\cM^\rd(\beta)_W=\Phi\inv(\tDelta\times W^\rd)$ and $\tDelta=\tau\inv_{\bcM_{0,5}}(\Delta)$, we have $\tau_{\bcM_{0,5}}(\st(C))\in\Delta$.
	Let $\Gamma^s$ be the convex hull of $v_1,\dots,v_4$ in $\Gamma$.
	By \cref{lem:components_of_the_domain_curve}(\ref{item:components:stabilization}), $\Gamma^s$ is a contraction (in terms of graph) of $\tau_{\bcM_{0,5}}(\st(C))$.
	So by the definition of $\Delta$, the path connecting $v_1,v_3$ and the path connecting $v_2,v_4$ in $\Gamma^s$ intersect, and the vertex $v_5$ lies in the intersection.
	Combining with the equality $h(v_5)=h^\rd(v^\rd)$ shown in the last paragraph, we deduce that
	\[\tau_{\cM^\rd}(\cM^\rd(\beta)_W)\subset\cM^\rd_{4+1}(B_\rf),\]
	where $\cM^\rd_{4+1}(B_\rf)$ is defined in \cref{const:tropical_moduli}.
	
	By \cref{prop:tropical_rigidity}, the subset $T^\rd_\rf\subset M^\rd_{4+1}(B_\rf)$ is a union of connected components.
	The tropicalization map $\tau_{\cM^\rd}$ is continuous by \cite[Theorem 8.1]{Yu_Tropicalization_2014}.
	Therefore, the substack $\cM^\rd(\beta,T)_W\subset\cM^\rd(\beta)_W$ is a union of connected components; in particular, it is a \kanal stack.
	Since the algebraic stack $\cM^\rd(\beta)$ is proper over $k$, by \cite[Proposition 6.4]{Porta_Yu_Higher_analytic_stacks_2014}, its analytification $\cM^\rd(\beta)\an$ is a proper \kanal stack.
	So the map
	\[\Phi\colon\cM^\rd(\beta)\an\to\bcM\an_{0,5}\times(Y')\an\]
	is proper.
	By base change, the map
	\[\Phi|_{\cM^\rd(\beta)_W}\colon \cM^\rd(\beta)_W\to\tDelta\times W^\rd\]
	is proper.
	Since the substack $\cM^\rd(\beta,T)_W\subset\cM^\rd(\beta)_W$ is a union of connected components, we deduce that the map
	\[\Phi|_{\cM^\rd(\beta,T)_W} \colon \cM^\rd(\beta,T)_W \to\tDelta\times W^\rd\]
	is also proper, completing the proof.
\end{proof}

\section{Smoothness of moduli spaces}\label{sec:smoothness}

In \cref{sec:analytic_constructions}, we have shown the properness of various evaluation maps.
In this section, we study the smoothness of moduli spaces.
The main result is \cref{prop:finite_etale}, which shows that after a small deformation of the target log Calabi-Yau surface, the evaluation maps become generically finite étale.
As a consequence, we obtain the positivity and integrality theorem.

\begin{prop}\label{prop:smoothness}
	Let $\cM(\hL^0,\beta)_0$ denote the substack of $\cM(\hL^0,\beta)$ consisting of stable maps $[C,(s_1,s_2,s_3),f\colon C\to Y']$ satisfying the following conditions:
	\begin{enumerate}
		\item The domain curve $C$ is irreducible.
		\item \label{item:smoothness:boundary} The preimage $f\inv(\partial Y')$ is a Cartier divisor supported at the two marked points $s_1$ and $s_2$.
	\end{enumerate}
	Then there exists a nonempty Zariski open subset $V\subset Y'$ over which the evaluation map
	\[\ev_3\colon\cM(\hL^0,\beta)_0\to Y'\]
	is smooth.
	
	Similarly, we define the substack $\cM^\rp(\hL^0,\beta)_0\subset\cM^\rp(\hL^0,\beta)$.
	Then there exists a nonempty Zariski open subset $V\subset\tY$ over which the evaluation map
	\[\ev_3\colon\cM^\rp(\hL^0,\beta)_0\to\tY\]
	is smooth.
\end{prop}
\begin{proof}
		Let us prove the statement for the stack $\cM(\hL^0,\beta)$.
	The same arguments work for the stack $\cM^\rp(\hL^0,\beta)$.
	Let $x$ be a closed point in $\cM(\hL^0,\beta)_0$ and let $[C,(s_1,s_2,s_3),f\colon C\to Y']$ be the corresponding stable map.
	Let $E$ denote the vector bundle $f^*T_{Y'}(-\log\partial Y')$.
	The derivative $d\ev_3$ of the evaluation map $\ev_3$ at $x$ is given by the map
	\[ H^0(C,E)\longrightarrow f^*(T_{Y'})_{s_3}.\]
	By Condition (\ref{item:smoothness:boundary}), the image of $s_3$ does not lie in the boundary $\partial Y'$.
	So we have
	\[f^*(T_{Y'})_{s_3}\simeq E_{s_3}.\]
	Assume that the derivative $d\ev_3$ is surjective.
	Then the vector bundle $E$ is globally generated at the point $s_3$.
	Since $C$ is a projective line, the vector bundle $E$ is semi-positive, i.e.\ it is isomorphic to a direct sum of line bundles $\cO(i)$ with $i\ge 0$.
	Therefore, the first cohomology $H^1(C,E)$ is zero.
	So by the dimension estimate in \cite[Proposition 5.3]{Keel_Rational_curves_1999}, the moduli space $\cM(\hL^0,\beta)_0$ is smooth at $x$.
	By the surjectivity of the derivative $d\ev_3$ again, we deduce that the evaluation map $\ev_3$ is smooth at $x$.
	Now let $\cM'$ be the set of closed points in $\cM(\hL^0,\beta)_0$ where the derivative $\rd \ev_3$ is not surjective.
	By \cite[Proposition III.10.6]{Hartshorne_Algebraic_geometry_1977}, the Zariski closure of the image $\ev_3(\cM')$ has dimension less than or equal to one.
	So we can take $V$ to be the complement, completing the proof of the proposition.
\end{proof}

\begin{lem}\label{lem:trivial_bundle}
	Let $C$ be a nodal curve of arithmetic genus 0.
	Let $E$ be a vector bundle on $C$ such that $c_1(E)\cdot C'=0$ for every irreducible component $C'$ of $C$.
	Then $E$ is globally generated at a point $p\in C$ if and only if $E$ is a trivial vector bundle.
\end{lem}
\begin{proof}
	Note that $C$ is a tree of projective lines.
	Let $C'\subset C$ be an irreducible component of $C$ and let $p\in C'$ be a point.
	Assume that $E$ is globally generated at $p$.
	Then $E|_{C'}$ is a direct sum of line bundles $\cO(i)$ with $i\ge 0$.
	Since $c_1(E)\cdot C'=0$, the vector bundle $E|_{C'}$ is trivial.
	Thus $H^0(C',E|_{C'})\to E_q$ is an isomorphism for all $q\in C'$.
	So the sections in $H^0(C,E)$ which generate $E_p$ also generate $E_q$ for all $q\in C'$, in particular for any nodal point of $C$ contained in $C'$.
	Now we can propagate the argument over the full tree and deduce that $E$ is a trivial vector bundle.
	The proof of the other direction is obvious.
\end{proof}

\begin{prop}\label{prop:smoothness_d}
	Let $\cM^\rd(\beta)_0$ denote the substack of $\cM^\rd(\beta)$ consisting of stable maps $[C,(s_1,\dots,s_5),f\colon C\to Y']$ satisfying the following conditions:
	\begin{enumerate}
		\item The domain curve $[C,(s_1,\dots,s_5)]$ is a stable 5-pointed curve.
		\item The preimage $f\inv(\partial Y')$ is a Cartier divisor supported at the first four marked points.
	\end{enumerate}
	Let $\cM^\rd_\sm(\beta)_0\subset\cM^d(\beta)_0$ denote the open substack consisting of stable maps such that $f^*T_{Y'}(-\log\partial Y')$ is a trivial vector bundle on $C$.
	The following hold:
	\begin{enumerate}
		\item The map
		\[\Phi\coloneqq(\mathrm{st},\ev_5)\colon\cM^\rd(\beta)_0\longrightarrow\bcM_{0,5}\times Y'\]
		is smooth over $\cM^\rd_\sm(\beta)_0$.
		\item Given any closed point $\mu\in\bcM_{0,5}$, let
		\[\Phi_\mu\colon\cM^\rd(\beta)_{0,\mu}\to Y'\]
		denote the restriction of the map $\Phi$ to the fibers over $\mu$.
		There exists a Zariski dense open subset $V_\mu\subset Y'$ such that
		\[\Phi_\mu\inv(V_\mu)\subset\cM^\rd_\sm(\beta)_0.\]
	\end{enumerate}
\end{prop}
\begin{proof}
	Fix any closed point $\mu\in\bcM_{0,5}$.
	Let $x\in\cM^\rd(\beta)_0$ be a closed point over $\mu$.
	Let $[C,(s_1,\dots,s_5),f\colon C\to Y']$ be the corresponding stable map.
	Let $E$ denote the vector bundle $f^*T_{Y'}(-\log\partial Y')$.
	The derivative $d\Phi_\mu$ at the point $x$ is given by the map
	\[H^0(C,E)\to f^*(T_{Y'})_{s_5}\simeq E_{s_5}.\]
	Assume that $x$ lies in $\cM^\rd_\sm(\beta)_0$.
	In this case, $E$ is a trivial vector bundle.
	Hence the derivative $d\Phi_\mu$ is surjective at $x$.
	Moreover, the first cohomology $H^1(C,E)$ vanishes.
	
	By \cite[Chapter II Theorem 1.7]{Kollar_Rational_curves_1996}, the dimension of $\cM^\rd(\beta)_0$ at the point $x$ is at least
	\[h^0(C,E)-h^1(C,E)+\dim\bcM_{0,5}=h^0(C,E)+\dim\bcM_{0,5}.\]
	Note that the dimension of the Zariski tangent space of $\cM^\rd(\beta)_{0,\mu}$ at $x$ is at most $h^0(C,E)$.
	Moreover, the dimension of the Zariski tangent space of $\cM^\rd(\beta)_0$ at $x$ is at most the dimension of the Zariski tangent space of $\cM^\rd(\beta)_{0,\mu}$ at $x$ plus the dimension of the Zariski tangent space of $\bcM_{0,5}$ at $m$.
	Therefore, the dimension of the Zariski tangent space of $\cM^\rd(\beta)_0$ at $x$ equals the dimension of the space $\cM^\rd(\beta)_0$ itself at $x$.
	So $\cM^\rd(\beta)_0$ is smooth at $x$.
	Furthermore, the surjectivity of the derivative $d\Phi_\mu$ at $x$ implies the surjectivity of the derivative $d\Phi$ at $x$.
	Thus the map $\Phi$ is smooth at $x$.
	This shows the first statement in the proposition.
	
	Next we drop the assumption that $x$ lies in $\cM^\rd_\sm(\beta)_0$.
	If the derivative $d\Phi_\mu$ is surjective at $x$, then the vector bundle $E$ is globally generated at the point $s_5$; hence by \cref{lem:trivial_bundle}, $E$ is trivial.
	In other words, $\cM^\rd_\sm(\beta)_{0,\mu}\subset\cM^\rd(\beta)_{0,\mu}$ is exactly the locus where the derivative $d\Phi_\mu$ is surjective.
	Hence by \cite[Proposition III.10.6]{Hartshorne_Algebraic_geometry_1977}, the Zariski closure of the image $\Phi_\mu\big(\cM^\rd(\beta)_{0,\mu}\setminus\cM^\rd_\sm(\beta)_{0,\mu}\big)$ has dimension less than the dimension of $Y'$.
	This shows the second statement in the proposition.
\end{proof}

\begin{prop}\label{prop:deformation}
	Given any Looijenga pair $(Y,D)$, there is a deformation $\cY$ of $Y$ preserving $D$ such that $\cY$ does not contain any proper curve disjoint from $D$.
\end{prop}
\begin{proof}
	Denote
	\[D^\perp\coloneqq\Set{\alpha\in\Pic(Y) | \alpha\cdot[D']=0\text{ for every irreducible component } D' \text{ of } D}.\]
	Let $\Pic^0(D)\subset\Pic(D)$ denote the subset consisting of line bundles on $D$ of degree 0 on every irreducible component of $D$.
	Fix a cyclic ordering of the irreducible components of $D$.
	By \cite[Lemma 2.1(1)]{Gross_Moduli_of_surfaces_2015}, it induces an identification $\Pic^0(D)=\GmC$.
	
	Let $\cY$ be a deformation of $Y$ preserving $D$.
	Restriction of line bundles gives a homomorphism
	\[\widetilde{\phi}_{\cY}\colon\Pic(\cY)\to\Pic(D).\]
	Since $\Pic(Y)\simeq H^2(Y,\Z)$, the lattice $\Pic(\cY)$ can be canonically identified with $\Pic(Y)$, together with the intersection form and the classes of the irreducible components of $D$.
	So $\widetilde{\phi}_{\cY}$ induces a homomorphism
	\[\phi_{\cY}\colon D^\perp\to\Pic^0(D)=\GmC.\]
	
	Let $\operatorname{Def}(Y,D)$ denote the versal deformation space of $Y$ preserving $D$.
	The construction above gives a local period mapping
	\[\phi\colon\operatorname{Def}(Y,D)\to\Hom(D^\perp,\GmC),\]
	which is a local analytic isomorphism (see \cite[Proposition 4.1]{Gross_Moduli_of_surfaces_2015} and \cite[II.2.5]{Looijenga_Rational_1981}).
	
	For each $\alpha\in D^\perp\setminus 0$, let $T_\alpha\subset\Hom(D^\perp,\GmC)$ be the subtorus consisting of $\psi\in\Hom(D^\perp,\GmC)$ such that $\psi(\alpha)=1$.
	Let $\cY$ be a deformation of $Y$ in the preimage
	\[\phi\inv\Biggl(\Hom(D^\perp,\GmC)\setminus\biggl(\bigcup_{\alpha\in D^\perp\setminus 0} T_\alpha\biggr)\Biggr).\]
	By construction, there is no nontrivial line bundle $\alpha$ on $\cY$ in $D^\perp$ whose restriction to $D$ is trivial.
	In particular, there are no proper curves in $\cY$ disjoint from $D$.
\end{proof}

By \cite[Theorem 5.4]{Yu_Enumeration_cylinders_2015}, the moduli stack responsible for the enumeration of holomorphic cylinders is a union of connected components of the moduli stack of stable maps (with certain incidence conditions).
Moreover, the virtual fundamental class we use is the restriction of the Gromov-Witten virtual fundamental class.
Therefore, the deformation invariance of Gromov-Witten invariants implies that the enumeration of holomorphic cylinders for a Looijenga pair $(Y,D)$ is invariant under deformations of $Y$ preserving $D$.
Thus, in virtue of \cref{prop:deformation}, we can assume from now on the following:

\begin{assumption} \label{assum:deformation}
	We assume that for the Looijenga pair $(Y,D)$, the projective surface $Y$ does not contain any proper curve disjoint from $D$.
\end{assumption}

The main point of the following two lemmas is that under the assumption above, the analytic stable maps with desired tropicalizations will have stable domains.

\begin{lem}\label{lem:stable_domain}
	Let $\cM(\hL^0,\beta)_0$ and $\cM^\rd(\beta)_0$ be as in Propositions \ref{prop:smoothness} and \ref{prop:smoothness_d} respectively.
	We have $\cM(\hL^0,\beta,T)\subset\cM(\hL^0,\beta)\an_0$ and $\cM^\rd(\beta,T)\subset\cM^\rd(\beta)\an_0$.
\end{lem}
\begin{proof}
	Under \cref{assum:deformation}, \cref{lem:components_of_the_domain_curve}(\ref{item:components:unstable}) implies that the pointed domain curve of any analytic stable map in $\cM(\hL^0,\beta,T)$ or $\cM^\rd(\beta,T)$ is stable.
	Therefore, combining with \cref{lem:components_of_the_domain_curve}(\ref{item:components:marked_points}), we obtain the two inclusions in the statement of the lemma.
\end{proof}

Similarly, we have the following:

\begin{lem}\label{lem:stable_domain_p}
	Let $p_{Y,1}\colon\tY\to\bbP^1_k$ denote the projection map induced by the map of simplicial cone complexes $\tB'\to\R=\R_{<0}\cup\{0\}\cup\R_{>0}$.
	Consider the composite map
	\[(p\an_{Y,1}\circ\ev_3)\colon\cM^\rp(\hL^0,\beta,T)\to(\bbP^1_k)\an.\]
	Let $\cM^\rp(\hL^0,\beta,T)_1\coloneqq(p\an_{Y,1}\circ\ev_3)\inv(\Gmk\an)$.
	Let $\cM^\rp(\hL^0,\beta)_0$ be as in \cref{prop:smoothness}.
	Then we have $\cM^\rp(\hL^0,\beta,T)_1\subset\cM^\rp(\hL^0,\beta)\an_0$.
\end{lem}
\begin{proof}
	Under \cref{assum:deformation}, \cite[Lemma 5.9(1.2)]{Gross_Mirror_Log_2011} implies that there exists an ample effective divisor $A$ on $Y'$ with support equal to $\partial Y'$.
	As $\tY$ is a toric blowup of $\bbP_k^1\times Y'$, we deduce that there also exists an ample effective divisor $\tA$ on $\tY$ with support equal to $\partial\tY$.
	
	Let $\ttau\colon (\tY\setminus\partial\tY)\an\to\tB$ denote the retraction map.
	Let $[C,(s_1,s_2,s_3),f\colon C\to\tY\an]$ be a stable map in $\cM^\rp(\hL^0,\beta,T)_1$.
	By definition, there is an irreducible component $C_0$ of $C$ with three marked points $\{\widebar s_1,\widebar s_2,\widebar s_3\}$ such that the restriction of $(\ttau\circ f)$ to the path $(\widebar s_1,\widebar s_2)\subset C_0$ (excluding the endpoints) is equal to
	\[\widetilde h^0\coloneqq(\iota,\hh^0)\colon (\hGamma^0)^\circ\to\R\times B=\tB,\]
	where $\iota\colon(\hGamma^0)^\circ\to\R$ is the homeomorphism preserving the \Zaffine structures.
		Therefore, by \cref{const:spbeta}(\ref{const:spbeta:intersect}), for any $\gamma\in\NE(\tY)$, we have
	\[\gamma\cdot\tA\le[f(C_0)]\cdot\tA.\]
	Since $[f(C)]\cdot\tA=\gamma\cdot\tA$ for some $\gamma\in\NE(\tY)$ and $\tA$ is ample, for any irreducible component $E$ of $C$ other than $C_0$, $f|_E$ is necessarily constant.
	This implies that the pointed domain curve $[C,(s_1,s_2,s_3)]$ is stable.
	Thus composing $f$ with the projection $p_{Y,2}\colon\tY\to Y'$, we obtain a stable map in $\cM(\hL^0,\beta,T)$.
	Now we conclude from \cref{lem:stable_domain}.
\end{proof}

\begin{lem} \label{lem:non-stacky}
	The \kanal stacks $\cM(\hL^0,\beta,T)$, $\cM^\rp(\hL^0,\beta,T)_1$ and $\cM^\rd(\beta,T)$ are \kanal spaces.
\end{lem}
\begin{proof}
	The stack $\cM(\hL^0,\beta)$ is an algebraic \DM stack over $k$, so it is a quotient of an étale groupoid in $k$-schemes of finite type.
	Then its analytification $\cM(\hL^0,\beta)\an$ is a quotient of an étale groupoid in \kanal spaces, so is the analytic domain $\cM(\hL^0,\beta,T)$ in $\cM(\hL^0,\beta)\an$.
	Note that a stable pointed rational curve does not have non-trivial automorphisms.
	So by \cref{lem:stable_domain}, the points in $\cM(\hL^0,\beta,T)$ do not have non-trivial automorphisms.
	Then we can write $\cM(\hL^0,\beta,T)$ as a quotient of an étale equivalence relation in \kanal spaces.
	Since the \kanal stack $\cM(\hL^0,\beta,T)$ is separated, it follows from \cite[Theorem 1.2.2]{Conrad_Non-archimedean_analytification_2009} that it is isomorphic to a \kanal space.
	Using \cref{lem:stable_domain_p} and \cref{lem:stable_domain}, the same argument shows that the \kanal stacks $\cM^\rp(\hL^0,\beta,T)_1$ and $\cM^\rd(\beta,T)$ are also \kanal spaces.
\end{proof}

We recall that a map of \kanal spaces $\phi\colon X'\to X$ is said to be \emph{finite} if there is an affinoid G-covering $\{V_i\}$ of $X$ such that every $\phi\inv(V_i)\to V_i$ is a finite map of $k$-affinoid spaces, i.e.\ the corresponding homomorphism of $k$-affinoid algebras is finite (see \cite[\S 1.3]{Berkovich_Etale_1993} for the G-topology).

We refer to \cite{Berkovich_Etale_1993,Ducros_Families_2011} for the notion of étale map of \kanal spaces.
For a finite étale map of \kanal spaces $\phi\colon X'\to X$, we have an affinoid G-covering $\{V_i\}$ of $X$ such that every $\phi\inv(V_i)\to V_i$ is a finite étale map of $k$-affinoid spaces, i.e.\ the corresponding homomorphism of $k$-affinoid algebras is finite étale.
Therefore, $\phi_*\cO_{X'}$ is a finite G-locally free $\cO_X$-module.
We define the \emph{degree} of the finite étale map $\phi$ to be the rank of $\phi_*\cO_{X'}$ viewed as a finite G-locally free $\cO_X$-module.

By \cref{lem:non-stacky}, all the moduli stacks involved in the following proposition and the following theorem are \kanal spaces.

\begin{prop}\label{prop:finite_etale}
	There exists three Zariski dense open subsets
	$V^0\subset W^0$, $V^\rp\subset W^\rp$ and $V^\rd\subset\tDelta\times W^\rd$, such that the three maps
	\begin{align*}
	&\ev_3\colon\cM(\hL^0,\beta,T)_W\to W^0\\
	&\ev_3\colon\cM^\rp(\hL^0,\beta,T)_W\to W^\rp\\
	&\Phi=(\mathrm{st},\ev_5)\colon\cM^\rd(\beta,T)_W\to\tDelta\times W^\rd
	\end{align*}
	considered in \cref{sec:analytic_constructions} are finite étale over $V^0$, $V^\rp$ and $V^\rd$ respectively.
	Moreover, we can require that the projection $V^\rd\to\tDelta$ is surjective.
	We will denote respectively by $\cM(\hL^0,\beta,T)_V$, $\cM^\rp(\hL^0,\beta,T)_V$ and $\cM^\rd(\beta,T)_V$ the preimages of $V^0, V^\rp$ and $V^\rd$ under the three maps considered above.
\end{prop}
\begin{proof}
	Let us prove for the first map.
	\cref{prop:smoothness} and \cref{lem:stable_domain} imply that there exists a Zariski dense open subset $V^0\subset W^0$ such that the first map is smooth over $V^0$.
	Using the notations in the proof of \cref{prop:smoothness}, the dimension of the space $\cM(\hL^0,\beta,T)_V$ is equal to the dimension of $H^0(C,E)$.
	Using the Riemann-Roch formula, we compute that
	\[\dim(H^0(C,E))=\chi(C,E)=\rank E + \deg E = 2+0 =2.\]
	So the dimension of $\cM(\hL^0,\beta,T)_V$ equals the dimension of $W^0$.
	Thus the map $\ev_3\colon\cM(\hL^0,\beta,T)_W\to W^0$ is étale over $V^0$.
	Moreover, the map $\ev_3\colon \cM(\hL^0,\beta,T)_W \to W^0$ is proper by \cref{prop:properness}.
	So it is finite over $V^0$ by \cite[Corollary 3.3.8]{Berkovich_Spectral_1990}.
	To conclude, the map $\ev_3\colon\cM(\hL^0,\beta,T)_W\to W^0$ is finite étale over $V^0$.
	
	For the second map, the same reasoning applies, using \cref{prop:smoothness}, \cref{lem:stable_domain_p} and \cref{prop:properness}.
	
	Now we consider the third map.
	By \cref{prop:properness_d}, the map
	\[\Phi\colon\cM^\rd(\beta,T)_W\to\tDelta\times W^\rd\]
	is proper.
	Let $\cM^\rd_\sm(\beta)_0\subset\cM^\rd(\beta)_0\subset\cM^\rd(\beta)$ be as in \cref{prop:smoothness_d}.
	By \cref{lem:stable_domain}, we have $\cM^\rd(\beta,T)\subset\cM^\rd(\beta)_0\an$.
	Let
	\[V^\rd\coloneqq\tDelta\times W^\rd\setminus\Phi\big(\cM^\rd(\beta,T)_W\setminus\cM^\rd_\sm(\beta)_0\an\big).\]
	The properness of $\Phi$ implies that $V^\rd\subset\tDelta\times W^\rd$ is Zariski open.
	\cref{prop:smoothness_d} shows that $\Phi$ is smooth over $V^\rd$ and that the projection $V^\rd\to\tDelta$ is surjective.
	Using the notations in the proof of \cref{prop:smoothness_d}, the dimension of $\cM^\rd(\beta,T)_V$ is equal to $h^0(C,E)+\dim\bcM_{0,5}$.
	Using the Riemann-Roch formula again, we see that the dimension of $\cM^\rd(\beta,T)_V$ is equal to the dimension of $\tDelta\times W^\rd$.
	Combining with smoothness, we conclude that $\Phi$ is étale over $V^\rd$.
	Finally, combining with properness, we conclude that $\Phi$ is finite étale over $V^\rd$.
\end{proof}

\begin{thm}\label{thm:degree}
	Let $\beta'$ denote the curve class associated to the extension from $L^0$ to $\hL^0$ (see \cref{const:extension}).
	Let $\beta^0\coloneqq\beta-\beta'$.
	The finite étale maps
	\begin{align*}
	&\ev_3\colon\cM(\hL^0,\beta,T)_V\to V^0, \text{ and}\\
	&\ev_3\colon\cM^\rp(\hL^0,\beta,T)_V\to V^\rp
	\end{align*}
	have the same degree.
	The degree is equal to the number $N(L^0,\beta^0)$ of holomorphic cylinders associated to the spine $L^0$ and the curve class $\beta^0$.
	As a consequence, the number $N(L^0,\beta^0)$ is a nonnegative integer.
\end{thm}

\begin{rem}
	Here $\beta$ is the curve class of the rational stable maps, while $\beta^0\coloneqq\beta-\beta'$ is the curve class of the holomorphic cylinders, because $\beta'$ represents the sum of curve classes of the two auxiliary disks glued to our cylinders in order to make them into rational curves.
	The notion of curve class of a curve with boundary is defined via specialization as in \cite[Definition 5.10]{Yu_Enumeration_cylinders_2015}.
\end{rem}

\begin{proof}[Proof of \cref{thm:degree}]
	Let $\cM(\hL^0,\beta)_0$ and $\cM^\rp(\hL^0,\beta)_0$ be as in \cref{prop:smoothness}.
	Recall that we have two projections $p_{Y,1}\colon\tY\to\bbP^1_k$ and $p_{Y,2}\colon\tY\to Y'$.
	Composing the evaluation map of the third marked point with the first projection, we get a natural map
	\[p_{Y,1}\circ\ev_3\colon \cM^\rp(\hL^0,\beta)_0\to\bbP^1_k.\]
	Composing the stable maps with the second projection, we get a natural map
	\[\pi_2\colon \cM^\rp(\hL^0,\beta)_0\to\cM(\hL^0,\beta)_0.\]
	Combining the two, we obtain a natural map
	\[\Psi\coloneqq(p_{Y,1}\circ\ev_3,\pi_2)\colon\cM^\rp(\hL^0,\beta)_0\to\bbP^1_k\times\cM(\hL^0,\beta)_0.\]
	By construction, we see that $\Psi$ is an isomorphism over $\Gmk\times\cM(\hL^0,\beta)_0$.
	Combining with Lemmas \ref{lem:stable_domain} and \ref{lem:stable_domain_p}, we deduce that the finite étale maps
	\begin{align*}
	&\ev_3\colon\cM(\hL^0,\beta,T)_V\to V^0, \text{ and}\\
	&\ev_3\colon\cM^\rp(\hL^0,\beta,T)_V\to V^\rp
	\end{align*}
	have the same degree.
	The construction of the number $N(L^0,\beta^0)$ in \cite{Yu_Enumeration_cylinders_2015}, as reviewed in \cref{sec:review_of_cylinder_counts}, involves a virtual fundamental class.
	Here in our situation, it follows from étaleness that the virtual fundamental class is equal to the fundamental class (see \cite[Example after Proposition 5.3]{Behrend_Intrinsic_normal_cone_1997}).
	Therefore, the number $N(L^0,\beta^0)$ is equal to the degree of the finite étale maps above.
	As a consequence, the number $N(L^0,\beta^0)$ is a nonnegative integer.
\end{proof}

\begin{rem} \label{rem:symmetry}
	Recall that we have the extended spine \[\hL^0=(\hGamma^0,(u^0_1,u^0_2),\hh^0\colon(\hGamma^0)^\circ\to B)\]
	associated to the spine $L^0=(\Gamma^0,(v^0_1,v^0_2),h^0\colon\Gamma^0\to B)$.
	Let $r\in\hGamma^0$ be any point in the interior of an edge of $\hGamma^0$, or any bounded vertex where the \Zaffine map $\hh^0$ is balanced.
	Fix a point $v\in V^0\subset X\an_k$.
	Let $v'$ be a general point in the fiber of $\tau\colon X\an_k\to B$ over $\hh^0(r)\in B$.
	Let $\cM(\hL^0,\beta,T)_v$ and $\cM(\hL^0,\beta,T)_{v'}$ denote respectively the fibers of the evaluation map
	\[\ev_3\colon\cM(\hL^0,\beta,T)\to (Y')\an\]
	over the points $v$ and $v'$ in $(Y')\an$.
	The symmetry property of \cite[Theorem 6.3]{Yu_Enumeration_cylinders_2015} implies that
	\[\deg\cM(\hL^0,\beta,T)_v=\deg\cM(\hL^0,\beta,T)_{v'}\]
	when $r$ is the vertex $v^0_2\in\hGamma^0$.
	In fact, the same proof as that of \cite[Theorem 6.3]{Yu_Enumeration_cylinders_2015} shows that the equality above always holds.
\end{rem}

\section{The gluing formula}\label{sec:gluing}

With the preparations in the previous sections, we will prove the gluing formula in this section.

Consider the moduli space $\bcM_{0,5}$ of 5-pointed rational stable curves over $k$.
We label the 5 marked points by $s^1_1, s^1_2, s^2_1, s^2_2$ and $s_5$.
The graph $G$ (resp.\ $G'$) and the labeling of its legs in \cref{const:graphs} give rise to a 0-dimensional stratum in $\bcM_{0,5}$ which we denote by $m$ (resp.\ $m'$).

Let $V^\rd$ be as in \cref{prop:finite_etale} and consider the projection $V^\rd\to\tDelta$.
Let $w$ be a point in the fiber over $m$ and let $w'$ be a point in the fiber over $m'$.
Let $\widebar w$ and $\widebar w'$ be respectively the images of $w$ and $w'$ under the projection $V^\rd\to W^\rd$.

Consider the map
\[\Phi=(\mathrm{st},\ev_5)\colon\cM^\rd(\beta,T)_V\to V^\rd.\]
Let $\cM^\rd(\beta,T)_{w}$ and $\cM^\rd(\beta,T)_{w'}$ be the two fibers over the points $w$ and $w'$ respectively.

\begin{prop}\label{prop:point}
	Let $[C,(s^1_1, s^1_2,s^2_1,s^2_2,s_5),f\colon C\to (Y')\an]$ be a stable map in $\cM^\rd(\beta,\allowbreak T)_{w}$.
	Let $C^5$ be the irreducible component of $C$ containing the marked point $s_5$.
	Then we have $f(C^5)=\widebar w$.
	The same applies to $w'$.
\end{prop}
\begin{proof}
	Since $w\in V^\rd$ lies in the fiber over $m\in\bcM_{0,5}$, the irreducible component $C^5$ does not contain the other marked points $s^1_1,s^1_2,s^2_1,s^2_2$.
	By \cref{lem:components_of_the_domain_curve}(\ref{item:components:marked_points}), the image $f(C^5)$ lies in the interior $(Y'\setminus\partial Y')\an\simeq X\an_k$.
	The map $f|_{C^5}\colon C^5\to X\an_k$ gives rise to a tropical curve $h\colon\Gamma\to B$, where $\Gamma$ is a \Zaffine tree and $h$ is a \Zaffine immersion that is balanced everywhere outside of $h\inv(O)$.
	
	We claim that $\Gamma$ must be a point.
	Suppose the contrary.
	Then since $h$ is an immersion, the image $h(\Gamma)$ is not a point.
	So by \cref{lem:twig}, the image $h(\Gamma)$ is contained in a ray $R$ in $B$ starting from $O\in B$ with rational slope.
	Let $b$ be the point in $h(\Gamma)$ that is farthest from $O$, and let $v$ be a vertex in $\Gamma$ that maps to $b$.
	Then $h$ cannot be balanced at $v$, which is a contradiction.
	
	Since $f(s_5)=\widebar w$, the claim implies that the image $h(\Gamma)$ is equal to the point $\tau(\widebar w)\in B$.
	Since the preimage $\tau\inv(\tau(\widebar w))$ is affinoid, it cannot contain any proper curve.
	So we must have $f(C^5)=\widebar w$, completing the proof.
	The same argument applies to $w'$.
\end{proof}

Consider the evaluation map of the third marked point
\[\ev_3\colon\cM(\hL^1,\beta,T)\to(Y')\an.\]
Let $\cM(\hL^1,\beta,T)_{\widebar w'}$ be the fiber over the point $\widebar w'\in W^\rd$.
Similarly, we define $\cM(\hL^2,\beta,T)_{\widebar w'}$, $\cM(\hL^3,\beta,T)_{\widebar w}$ and $\cM(\hL^4,\beta,T)_{\widebar w}$.

By the construction of the subset $V^\rd\subset\tDelta\times W^\rd$ in \cref{prop:finite_etale}, the space $\cM(\hL^2,\beta,T)_{\widebar w'}$ (resp. $\cM(\hL^3,\beta,T)_{\widebar w}$, $\cM(\hL^4,\beta,T)_{\widebar w}$) is finite étale over the point $\widebar w'$ (resp.\ $\widebar w, \widebar w$).

For $i=1,\dots,3$, let $(\beta^i)'$ denote the curve class associated to the extension of the spine from $L^i$ to $\hL^i$ (see \cite[Definition 4.10]{Yu_Enumeration_cylinders_2015}).

\begin{prop}\label{prop:degree}
	We have the following equalities:
	\begin{align*}
		& \deg\cM(\hL^1,\beta,T)_{\widebar w'} = N(L^1,\beta-(\beta^1)'),\\
		& \deg\cM(\hL^2,\beta,T)_{\widebar w'} = N(L^2,\beta-(\beta^2)'),\\
		& \deg\cM(\hL^3,\beta,T)_{\widebar w} = N(L^3,\beta-(\beta^3)').
	\end{align*}
\end{prop}
\begin{proof}
	Let $v$ be a general point in the fiber of the map $\tau\colon X\an_k\to B$ over $h^1(v^1_1)\in B$.
	Let $\cM(\hL^1,\beta,T)_{v}$ (resp.\ $\cM(\hL^3,\beta,T)_{v}$) be the fiber of the evaluation map
	\[\ev_3\colon\cM(\hL^1,\beta,T)\to(Y')\an \qquad \text{(resp.\ $\ev_3\colon\cM(\hL^3,\beta,T)\to(Y')\an$)}\]
	over the point $v$.
	From \cref{thm:degree}, replacing $L^0$ with $L^1, L^2, L^3$ respectively, we deduce the following equalities
	\begin{align*}
	& \deg\cM(\hL^1,\beta,T)_{v} = N(L^1,\beta-(\beta^1)'),\\
	& \deg\cM(\hL^2,\beta,T)_{\widebar w'} = N(L^2,\beta-(\beta^2)'),\\
	& \deg\cM(\hL^3,\beta,T)_{v} = N(L^3,\beta-(\beta^3)').
	\end{align*}
	By the symmetry property in \cite[\S 6]{Yu_Enumeration_cylinders_2015}, as we recalled in \cref{rem:symmetry}, we have
	\begin{align*}
	&\deg\cM(\hL^1,\beta,T)_{v} = \deg\cM(\hL^1,\beta,T)_{\widebar w'},\\
	& \deg\cM(\hL^3,\beta,T)_{v} = \deg\cM(\hL^3,\beta,T)_{\widebar w}.
	\end{align*}
	We conclude the proof by combining all the equalities above.
\end{proof}

Lemmas \ref{lem:automorphism1}, \ref{lem:automorphism2} and \cref{prop:glue_annuli} will serve the proof of \cref{prop:straight_spine}.

\begin{lem} \label{lem:automorphism1}
	Let $\Gmk^n\coloneqq\Spec k[X_1^{\pm 1},\dots,X_n^{\pm 1}]$.
	Let $\sigma\colon(\Gmk^n)\an\to\R^n$ be the continuous map given by taking coordinate-wise valuations.
	Let $\Omega\subset\R^n$ be a rational convex polyhedron, i.e., it is given by finitely many linear inequalities with rational coefficients and is compact.
		Assume that $\Omega$  has nonempty interior.
	Let $T_\Omega$ denote the affinoid algebra of the affinoid space $\sigma\inv(\Omega)$.
	Let $\phi$ be an automorphism of $\sigma\inv(\Omega)$ preserving the map $\sigma$.
	Then for any $\gamma\in\Z^n$, we have $\phi^*(X^\gamma)=c_\gamma\cdot X^\gamma (1+f_\gamma)$, where $c_\gamma$ is an element in $k$ of norm 1, and $f_\gamma$ is an element in $T_\Omega$ of spectral norm less than 1.
\end{lem}
\begin{proof}
	Write $\phi^*(X^\gamma)=\sum_{\nu\in\Z^n}c_\nu X^\nu$ with $c_\nu\in k$.
	Since $\phi$ preserves the map $\sigma$, $\phi^*(X^\gamma)$ has no zeros in $\sigma\inv(\Omega)$.
	We claim that there exists $\nu_0\in\Z^n$ such that $\abs{(c_{\nu_0} X^{\nu_0})(x)}>\abs{(c_\nu X^\nu)(x)}$ for all $\nu\neq\nu_0$ and all rigid points (i.e.\ points associated to maximal ideals) $x\in\sigma\inv(\Omega)$.
	Otherwise, there is a rigid point $x\in\sigma\inv(\Omega)$, $\nu_1,\nu_2\in\Z^n$ such that $\abs{(c_{\nu_1} X^{\nu_1})(x)}=\abs{(c_{\nu_2} X^{\nu_2})(x)}\ge\abs{(c_\nu X^\nu)(x)}$ for all $\nu\in\Z^n$.
	Fixing the first $(n-1)$ coordinates of $x$ and varying the last one, by \cite[Lemma 9.7.1/1]{Bosch_Non-archimedean_1984}, we see that $\phi^*(X^\gamma)$ is not invertible, which is a contradiction.
	So we have justified our claim.
	Using again the assumption that $\phi$ preserves the map $\sigma$, we see that $\nu_0=\gamma$ and $\abs{c_\gamma}=1$, completing the proof.
\end{proof}

\begin{lem} \label{lem:automorphism2}
	Let $\bbA_k^n\coloneqq\Spec k[X_1,\dots,X_n]$.
	Let $\widebar\sigma\colon(\bbA_k^n)\an\to(-\infty,+\infty]^n$ be the continuous map given by taking coordinate-wise valuations.
	Let $\oOmega\subset(-\infty,+\infty]^n$ be a rational convex polyhedron, i.e., it is given by finitely many linear inequalities with rational coefficients and is compact.
	Assume that $\widebar\Omega$  has nonempty interior.
	Let $T_{\oOmega}$ be the affinoid algebra of the affinoid space $\widebar\sigma\inv(\oOmega)$.
	Let $f_1,\dots,f_n\in T_{\oOmega}$ be $n$ elements of spectral norm less than 1.
	Let $\varphi$ be the endomorphism of $T_{\oOmega}$ given by $X_i\mapsto X_i(1+f_i)$.
	Then $\varphi$ is an automorphism.
	It induces an automorphism of $\widebar\sigma\inv(\oOmega)$ preserving the map $\widebar\sigma$.
\end{lem}
\begin{proof}
	Decompose $\varphi=\id+B$ as an operator on the Banach space $T_{\oOmega}$.
		Since $f_i$ has spectral norm less than 1, we see that $B$ has operator norm less than 1.
	Therefore, the operator $\varphi$ has an inverse $\varphi\inv=\sum_{j=0}^\infty (-1)^j B^j$.
	Since $\varphi$ preserves the multiplication on $T_{\oOmega}$, so does $\varphi\inv$.
	We conclude that $\varphi$ is an automorphism of the affinoid algebra $T_{\oOmega}$ and it induces an automorphism of the affinoid space $\widebar\sigma\inv(\oOmega)$ preserving the map $\widebar\sigma$.
\end{proof}

\begin{prop}\label{prop:glue_annuli}
	Let $n\in\Z_{\ge 0}$, $r_1,\dots,r_n,r'_1,\dots,r'_n,p,q,s,t\in\sqrt{\abs{k}}\cup\{+\infty\}$, (the notation $a\in\sqrt{\abs{k}}$ means that some positive integer power of $a$ lies in $\abs{k}$).
	Assume $0<r'_i<r_i<+\infty$, $0\le p<s<t<q\le+\infty$.
	Let $R\coloneqq\prod_{i=1}^n[r'_i,r_i]\subset\R^n$.
	Let $X$ be a \kanal space and $\rho\colon X\to R\times[p,q]$ a continuous map.
	Let $\pi\colon(\Spec k[T_1^{\pm 1},\dots,T_n^{\pm 1},S^{\pm 1}])\an\to(0,+\infty)^{n+1}$ be the continuous map given by taking coordinate-wise norms.
	Let $\widebar\pi\colon((\bbP^1_k)\an)^{n+1}\to[0,+\infty]^{n+1}$ be the continuous extension of $\pi$.
	Assume that $\rho\inv(R\times[p,t])$ and $\rho\inv(R\times[s,q])$ are affinoid subdomains of $X$, and we have isomorphisms
	\[\begin{tikzcd}[column sep=small]
	\rho\inv(R\times[p,t])\arrow{rr}{\alpha}[swap]{\sim} \arrow{rd}[swap]{\rho} && \widebar\pi\inv(R\times[p,t])\arrow{ld}{\widebar\pi}\\
	& R\times[p,t]
	\end{tikzcd}\]
	and
	\[\begin{tikzcd}[column sep=small]
	\rho\inv(R\times[s,q])\arrow{rr}{\beta}[swap]{\sim} \arrow{rd}[swap]{\rho} && \widebar\pi\inv(R\times[s,q])\arrow{ld}{\widebar\pi}\\
	& R\times[s,q]
	\end{tikzcd}.\]
	Then we have an isomorphism
	\[\begin{tikzcd}[column sep=small]
	X\arrow{rr}{\sim} \arrow{rd}[swap]{\rho} && \widebar\pi\inv(R\times[p,q])\arrow{ld}{\widebar\pi}\\
	& R\times[p,q]
	\end{tikzcd}.\]
\end{prop}
\begin{proof}
	Let $\cA,\cB,\cC$ be the affinoid algebras of the affinoid spaces $\rho\inv(R\times[s,t]), \rho\inv(R\times[p,t]), \rho\inv(R\times[s,q])$ respectively.
	Let $\norm{\cdot}_\cA$, $\norm{\cdot}_\cB$, $\norm{\cdot}_\cC$ denote the spectral norms on $\cA,\cB,\cC$ respectively.
	Assume that the isomorphisms $\alpha$ and $\beta$ are given by the following isomorphisms of affinoid algebras:
	\begin{gather*}
	k\langle r_1\inv T_1,r'_1 T_1\inv,\dots,r_n\inv T_n,r'_n T_n\inv,t\inv S,pS\inv\rangle\xrightarrow{\ \phi\ }\cB,\quad T_i\mapsto f_i,\ S\mapsto f,\\
	k\langle r_1\inv T_1,r'_1 T_1\inv,\dots,r_n\inv T_n,r'_n T_n\inv,s S',q\inv (S')\inv\rangle\xrightarrow{\ \psi\ }\cC,\quad T_i\mapsto g_i,\ S'\mapsto g.
	\end{gather*}
	In the case $p=0$, we remove the term $p S\inv$ from the expression above.
	In the case $q=+\infty$, we remove the term $q\inv (S')\inv$ from the expression above.
	Consider the isomorphism
	\[\cB\langle s f\inv\rangle\to\cC\langle t\inv g\inv\rangle,\quad f_i\mapsto g_i,\  f\mapsto g\inv.\]
	Up to multiplying $g_1,\dots,g_n$ and $g$ by elements in $k$ of norm 1, it follows from \cref{lem:automorphism1} that we can assume that $\norm{f_i-g_i}_\cA<\norm{f_i}_\cA$ and $\norm{f-g\inv}_\cA<\norm{f}_\cA$.
	Now we proceed as in the proof of \cite[Lemma 3.2]{Temkin_Local_2000}.
	Let $\cB_+$ and $\cC_+$ be the subspaces of $\cB$ and $\cC$ consisting of elements of the form $\sum_{\nu\in\Z^n,j\ge 0}\lambda_{\nu,j}f_1^{\nu_1}\cdots f_n^{\nu_n} f^j$ and $\sum_{\nu\in\Z^n,j\ge 0}\mu_{\nu,j}g_1^{\nu_1}\cdots g_n^{\nu_n} g^j$ respectively.
	It is shown in the proof of \cite[Lemma 3.2]{Temkin_Local_2000} that each element $a\in\cA$ can be decomposed as $a=b+c$ with $b\in\cB_+$, $c\in\cC_+$, and $\norm{b}_\cB,\norm{c}_\cC\le\norm{a}_\cA$.
	So we have $f_i-g_i=b_i+c_i$ and $f-g\inv=b+c$, where $b_i,b\in\cB_+$, $c_i,c\in\cC_+$, $\norm{b_i}_\cB,\norm{c_i}_\cC\le\norm{f_i-g_i}_\cA$ and $\norm{b}_\cB,\norm{c}_\cC\le\norm{f-g\inv}_\cA$.
	By \cref{lem:automorphism2}, the homomorphisms
	\[
	k\langle r_1\inv T_1,r'_1 T_1\inv,\dots,r_n\inv T_n,r'_n T_n\inv,t\inv S,pS\inv\rangle\xrightarrow{\ \phi\ }\cB,\quad T_i\mapsto f_i-b_i,\ S\mapsto f-b,\]
	\begin{multline*}
	k\langle r_1\inv T_1,r'_1 T_1\inv,\dots,r_n\inv T_n,r'_n T_n\inv,s S',q\inv (S')\inv\rangle\xrightarrow{\ \psi\ }\cC,\\ T_i\mapsto g_i+c_i,\ S'\mapsto (g\inv+c)\inv=g(1+gc)\inv
	\end{multline*}
	are isomorphisms.
	So we deduce the isomorphism claimed in the proposition.
\end{proof}

\begin{prop}\label{prop:straight_spine}
	Let $\beta^4=(\beta^1)'+(\beta^2)'-(\beta^3)'$.
	We have
	\[\deg\cM(\hL^4,\beta,T)_{\widebar w}=
	\begin{cases}
		1 & \text{for } \beta=\beta^4,\\
		0 & \text{otherwise.}
	\end{cases}\]
\end{prop}
\begin{proof}
	For any rigid point $w^4\in\tY\an$, let $\cM^\rp(\hL^4,\beta,T)_{w^4}$ be the fiber of the evaluation map
	\[\ev_3\colon\cM^\rp(\hL^4,\beta,T)\to\tY\an\]
	over $w^4$.
	By \cref{prop:finite_etale} and \cref{thm:degree}, replacing $\hL^0$ with $\hL^4$, there exists a rigid point $w^4\in\Gmk\an\times W^\rd\subset\tY\an$ such that $\cM^\rp(\hL^4,\beta,T)_{w^4}$ is finite étale over $w^4$ and its degree equals the degree of $\cM(\hL^4,\beta,T)_{\widebar w}$.
	
	Recall the notation
	\[\hL^4=[\hGamma^4,(u^4_1,u^4_2),\ \hh^4\colon(\hGamma^4)^\circ\to B].\]
	Let $\tB$ and $\tB'$ be as in \cref{const:tB}.
	For each $d$-dimensional simplicial cone in $\tB'$, we compactify it as $[0,+\infty)^d\subset[0,+\infty]^d$.
	The gluing of all such compactifications gives a natural compactification of $\tB'$, which we denote by $\bB$.
	Note that $\bB$ is isomorphic to the closure of $\tB'$ inside $\tY\an$ (see \cite[Remark 4.12]{Gubler_Skeletons_2014}).
	The retraction map $\ttau\colon(\tY\setminus\partial\tY)\an\to\tB'$ extends to a retraction map $\sigma\colon\tY\an\to\bB$ (see \cite[Theorem 4.13]{Gubler_Skeletons_2014}).
	
	Let
	\[h\coloneqq(\iota,\hh^4)\colon(\hGamma^4)^\circ\to\R\times B=\tB,\]
	where the map $\iota\colon(\hGamma^4)^\circ\to\R$ is the homeomorphism preserving the \Zaffine structures.
	Let $x_0$ be a point in $\hGamma^4$ and let $\gamma$ be a nonzero tangent vector at $x_0$.
	Let $\delta$ and $\epsilon$ be two small rational tangent vectors at $h(x_0)\in\tB$, so that together with $h(\gamma)$, they generate the tangent space of $\tB$ at $h(x_0)$.
	
	Let
	\[U\coloneqq\Set{x+a\delta_x+b\epsilon_x | x\in(\hGamma^4)^\circ,\ a,b\in[0,1]}\subset\tB,\]
	where $\delta_x$ and $\epsilon_x$ denote respectively the parallel transport of $\delta$ and $\epsilon$ along $h((\hGamma^4)^\circ)$ to the point $h(x)$.
	We choose $\delta$ and $\epsilon$ sufficiently small , so that $U$ is disjoint from $\R\times\{O\}\subset\tB$.
	
	Let $\oU$ be the closure of $U$ in $\bB$ and let $\tU\coloneqq\sigma\inv(\oU)\subset\tY\an$.
	Note that the retraction map $\sigma\colon\tY\an\to\bB$, outside $[-\infty,+\infty]\times\{O\}\subset\bB$, locally on the target, is modeled on the coordinate-wise valuation map $(((\bbP^1_k)\an))^3\to[-\infty,+\infty]^3$.
	Therefore, by \cref{prop:glue_annuli}, the analytic domain $\tU\subset\tY\an$ is isomorphic to the product of a projective line $(\bbP^1_k)\an$ with a 2-dimensional polyannulus.
	By construction, for any stable map in $\cM^\rp(\hL^4,\beta,T)_{w^4}$, its image lies in $\tU$, and its degree over $(\bbP^1_k)\an$ is equal to 1.
	So its curve class lies in $s^\rp(\hL^4,\beta^4)$.
	Moreover, the product structure of $\tU$ implies that when $\beta=\beta^4$, the degree of $\cM^\rp(\hL^4,\beta,T)_{w^4}$ over $w^4$ is 1.
	Finally we conclude by \cref{thm:degree}.
\end{proof}

Let $\bcM_{0,3}(\widebar w)$ denote the stack of 3-pointed rational stable maps into the point $\widebar w$.
The stack $\bcM_{0,3}(\widebar w)$ is just a point.
Similarly, we denote $\bcM_{0,3}(\widebar w')$ for the point $\widebar w'$.
Let $\beta^\rd\coloneqq\beta^3+(\beta^1)'+(\beta^2)'$.
It follows from \cref{prop:point} that we have two isomorphisms
\begin{align}
\label{eq:isom_w'}
&\cM^\rd(\beta^\rd,T)_{w'}\simeq\coprod_{\hbeta^1+\hbeta^2=\beta^\rd}\cM(\hL^1,\hbeta^1,T)_{\widebar w'}\times\bcM_{0,3}(\widebar w')\times\cM(\hL^2,\hbeta^2,T)_{\widebar w'},\\
\label{eq:isom_w}
&\cM^\rd(\beta^\rd,T)_{w}\simeq\coprod_{\hbeta^3+\hbeta^4=\beta^\rd}\cM(\hL^3,\hbeta^3,T)_{\widebar w}\times\bcM_{0,3}(\widebar w)\times\cM(\hL^4,\hbeta^4,T)_{\widebar w}.
\end{align}

By Propositions \ref{prop:degree} and \ref{prop:straight_spine}, we deduce from \eqref{eq:isom_w'} and \eqref{eq:isom_w} the following equalities:
\begin{align}
\deg\cM^\rd(\beta^\rd,T)_{w'} & = \sum_{\hbeta^1+\hbeta^2=\beta^\rd} \deg\cM(\hL^1,\hbeta^1,T)_{\widebar w'}\cdot \deg\cM(\hL^2,\hbeta^2,T)_{\widebar w'}\nonumber\\
& = \sum_{\hbeta^1+\hbeta^2=\beta^\rd} N(L^1,\hbeta^1-(\beta^1)') \cdot N(L^2,\hbeta^2-(\beta^2)')\nonumber\\
& = \sum_{\beta^1+\beta^2=\beta^3} N(L^1,\beta^1)\cdot N(L^2,\beta^2), \label{eq:equality_w'}\\
\deg\cM^\rd(\beta^\rd,T)_w & = \sum_{\hbeta^3+\hbeta^4=\beta^\rd} \deg\cM(\hL^3,\hbeta^3,T)_{\widebar w}\cdot\deg\cM(\hL^4,\hbeta^4,T)_{\widebar w}\nonumber\\
& = \deg\cM(\hL^3,\beta^3+(\beta^3)',T)_{\widebar w}\nonumber\\
& = N(L^3,\beta^3). \label{eq:equality_w}
\end{align}

By \cref{prop:finite_etale} and the connectedness of $V^\rd$, we have $\deg\cM^\rd(\beta^\rd,T)_w=\deg\cM^\rd(\beta^\rd,T)_{w'}$.
Combining \eqref{eq:equality_w'} and \eqref{eq:equality_w}, we conclude the proof of \cref{thm:gluing_formula}.

\bibliographystyle{plain}
\bibliography{dahema}

\end{document}